\pdfoutput=1
\documentclass{gtart_a}
\input{diagxy}
\let\to\rightarrow
\let\barrsquare\square
\let\square\undefined
\usepackage{pinlabel}

\title{Singular fibers of stable maps and signatures of 4--manifolds}

\author{Osamu Saeki}
\givenname{Osamu}
\surname{Saeki}
\address{Faculty of Mathematics\\
Kyushu University\\\newline
Hakozaki\\
Fukuoka 812-8581\\
Japan}
\email{saeki@math.kyushu-u.ac.jp}
\urladdr{http://www.math.kyushu-u.ac.jp/~saeki/}

\author{Takahiro Yamamoto}
\givenname{Takahiro}
\surname{Yamamoto}
\address{Department of Mathematics\\
Hokkaido University\\\newline
Sapporo 060-0810\\
Japan}
\email{taku\_chan@math.sci.hokudai.ac.jp}
\urladdr{}

\volumenumber{10}
\issuenumber{}
\publicationyear{2006}
\papernumber{11}
\lognumber{0504}
\startpage{359}
\endpage{399}

\doi{10.2140/gt.2006.10.359}
\MR{}
\Zbl{}

\arxivreference{} 
\arxivpassword{}   


\received{8 October 2004}
\revised{22 February 2006}
\accepted{12 January 2006}
\published{7 April 2006}
\publishedonline{7 April 2006}
\proposed{Yasha Eliashberg}
\seconded{Shigeyuki Morita, Rob Kirby}
\corresponding{}
\editor{}
\version{}

\let\xysavmatrix\xymatrix
\def\xymatrix{\disablesubscriptcorrection\xysavmatrix}
\AtBeginDocument{\let\bar\wbar\let\tilde\wtilde}
\numberwithin{equation}{section}

\makeatletter
\def\cnewtheorem#1[#2]#3{\newtheorem{#1}{#3}[section]
\expandafter\let\csname c@#1\endcsname\c@thm}

\newtheorem{thm}{Theorem}[section]
\cnewtheorem{cor}[thm]{Corollary}
\cnewtheorem{lem}[thm]{Lemma}
\cnewtheorem{prop}[thm]{Proposition}

\theoremstyle{definition}
\cnewtheorem{dfn}[thm]{Definition}

\theoremstyle{remark}
\cnewtheorem{rem}[thm]{Remark}

\makeatother

\newcommand{\rank}{\mathop{\mathrm{rank}}\nolimits}
\newcommand{\Int}{\mathop{\mathrm{Int}}\nolimits}
\newcommand{\id}{\mathop{\mathrm{id}}\nolimits}
\newcommand{\Diff}{\mathop{\mathrm{Diff}}\nolimits}

\begin{document}

\begin{asciiabstract}

\end{asciiabstract}

\begin{abstract} 
We show that for a $C^\infty$ stable map of an oriented 4--manifold
into a 3--manifold, the algebraic number of singular fibers of a
specific type coincides with the signature of the source 4--manifold.
\end{abstract}

\maketitle

\section{Introduction}\label{section1}

In \cite{Saeki04} the first named author developed the theory of singular
fibers of generic differentiable maps between manifolds of negative
codimension.  Here, the \emph{codimension} of a map $f \co M \to N$
between manifolds is defined to be $k = \dim N - \dim M$.  For $k \geq
0$, the fiber over a point in $N$ is a discrete set of points, as long as
the map is generic enough, and we can study the topology of such maps by
using the well-developed theory of multi-jet spaces (see, for example,
the article \cite{MatherV} by Mather).  However, in the case where $k <
0$, the fiber over a
point is no longer a discrete set, and is a complex of positive dimension
$-k$ in general.  This means that the theory of multi-jet spaces is not
sufficient any more, and in \cite{Saeki04} we have seen that the topology
of singular fibers plays an essential role in such a study.

In \cite{Saeki04}, as an explicit and important example of the theory
of singular fibers, $C^\infty$ stable maps of closed orientable
$4$--manifolds into $3$--manifolds were studied and their singular
fibers were completely classified up to the natural equivalence relation,
called the $C^\infty$ (or $C^0$) right-left equivalence (for a precise
definition, see \fullref{section2} of the present paper).  Furthermore, it
was proved that the number of singular fibers of a specific type (in the
terminology of \cite{Saeki04}, singular fibers of type $\mathrm{III}^8$)
of such a map is congruent modulo two to the Euler characteristic of
the source $4$--manifold (see \cite[Theorem~5.1]{Saeki04} and also
\fullref{cor:Euler} of the present paper).

In this paper, we will give an ``integral lift" of this modulo two
Euler characteristic formula.  More precisely, we consider $C^\infty$
stable maps of \emph{oriented} $4$--manifolds into $3$--manifolds, and
we give a sign $+1$ or $-1$ to each of its $\mathrm{III}^8$ type fiber,
using the orientation of the source $4$--manifold.  Then we show that
the algebraic number of $\mathrm{III}^8$ type fibers coincides with the
signature of the source oriented $4$--manifold (\fullref{thm:main}).

For certain Lefschetz fibrations, similar signature formulas have
already been proved by Matsumoto \cite{Matsumoto1,Matsumoto2}, Endo
\cite{Endo}, etc.  Our formula can be regarded as their analogue from
the viewpoint of singularity theory of generic differentiable maps.
The most important difference between Lefschetz fibrations and generic
differentiable maps is that not all manifolds can admit a Lefschetz
fibration, while every manifold admits a generic differentiable map.
(In fact, a single manifold admits a lot of generic differentiable maps.)
Furthermore, it is known that similar signature formulas do not hold
for arbitrary Lefschetz fibrations, since there exist oriented surface
bundles over oriented surfaces with nonzero signatures (see Meyer \cite{Meyer}).
In this sense, our formula is more general (see \fullref{rem:Lef}).
Our proof of the formula is based on the abundance of such generic maps
in some sense.

More precisely, our proof of the formula goes as follows.  We first
define the notion of a chiral singular fiber (for a precise definition,
see \fullref{section2}).  Roughly speaking, if a fiber can be transformed
to its ``orientation reversal" by an orientation preserving homeomorphism
of the source manifold, then we call it an achiral fiber; otherwise, a
chiral fiber.  On the other hand, we classify singular fibers of proper
$C^\infty$ stable maps of orientable $5$--manifolds into $4$--manifolds by
using methods developed in \cite{Saeki04}.  Then, for proper $C^\infty$
stable maps of oriented $4$--manifolds into $3$--manifolds, and those of
oriented $5$--manifolds into $4$--manifolds, we determine those singular
fibers in the classification lists which are chiral.  Furthermore, for
each chiral singular fiber that appears discretely, we define its sign
$(= \pm 1)$.

Let us consider two $C^\infty$ stable maps of $4$--manifolds into a
$3$--manifold which are oriented bordant.  Then by using a generic bordism
between them, which is a generic differentiable map of a $5$--manifold
into a $4$--manifold, and by looking at the $\mathrm{III}^8$--fiber locus
in the target $4$--manifold, we show the oriented bordism invariance of
the algebraic numbers of $\mathrm{III}^8$ type fibers of the original
stable maps of $4$--manifolds.  Finally, we verify our formula for
an explicit example of a stable map of an oriented $4$--manifold with
signature $+1$. (In fact, this final step is not so easy and needs a
careful analysis.)  Combining all these, we will prove our formula.

The paper is organized as follows.  In \fullref{section2} we give
some fundamental definitions concerning singular fibers of generic
differentiable maps, among which is the notion of a chiral singular
fiber.  In \fullref{section3} we recall the classification of singular
fibers of proper $C^\infty$ stable maps of orientable $4$--manifolds into
$3$--manifolds obtained in \cite{Saeki04}. In \fullref{section4} we present
the classification of singular fibers of proper $C^\infty$ stable maps
of orientable $5$--manifolds into $4$--manifolds. In \fullref{section5}
we determine those singular fibers in the classification lists which
are chiral. Furthermore, for each chiral singular fiber that appears
discretely, we define its sign by using the orientation of the source
manifold. In \fullref{section6} we prove the oriented bordism invariance of
the algebraic number of $\mathrm{III}^8$ type fibers. This is proved by
looking at the adjacencies of the chiral singular fiber loci in the target
manifold.  In \fullref{section7} we investigate the explicit example of a
$C^\infty$ stable map of an oriented $4$--manifold into a $3$--manifold
constructed in \cite{Saeki04}.  In order to calculate the signature of
the source $4$--manifold, we will compute the self-intersection number
of the surface of definite fold points by using normal sections coming
from the surface of indefinite fold points.  This procedure needs some
technical details so that this section will be rather long.  By combining
the result of \fullref{section6} with the computation of the example,
we prove our main theorem.  Finally in \fullref{section8}, we define
the universal complex of chiral singular fibers for proper $C^\infty$
stable maps of $5$--manifolds into $4$--manifolds and compute its third
cohomology group.  This will give an interpretation of our formula from
the viewpoint of the theory of singular fibers of generic differentiable
maps as developed in \cite{Saeki04}.

Throughout the paper, all manifolds and maps are differentiable of class
$C^\infty$.  The symbol ``$\cong$" denotes an appropriate isomorphism
between algebraic objects. For a space $X$, the symbol ``$\id_X$"
denotes the identity map of $X$.

The authors would like to express their thanks to Andr\'as Sz\H{u}cs
for drawing their attention to the work of Conner--Floyd, and to Go-o
Ishikawa for his invaluable comments and encouragement.  They would
also like to thank the referee for helpful comments.  The first named
author has been supported in part by Grant-in-Aid for Scientific Research
(No.~16340018), Japan Society for the Promotion of Science.

\section{Preliminaries}\label{section2}

Let us begin with some fundamental definitions.  For some of the
definitions, refer to \cite{Saeki04}.

\begin{dfn}\label{dfn:diffeo}
Let $M_i$ be smooth manifolds and $A_i \subset M_i$
be subsets, $i = 0, 1$. A continuous map $g \co A_0 \to A_1$
is said to be \emph{smooth} 
if for every point $q \in A_0$, there
exists a smooth map $\tilde{g} \co V \to M_1$ defined on a neighborhood
$V$ of $q$ in $M_0$ such that $\tilde{g}|_{V \cap A_0} = g|_{V \cap A_0}$.
Furthermore, a smooth map $g \co A_0 \to A_1$ is a \emph{diffeomorphism}
if it is a homeomorphism and its inverse is also 
smooth. When there exists a diffeomorphism
between $A_0$ and $A_1$, we say that they
are \emph{diffeomorphic}.
\end{dfn}

\begin{dfn}\label{dfn:equivalence}
Let $f_i \co M_i \to N_i$ be smooth maps and take points
$y_i \in N_i$, $i=0,1$. We say that the fibers over 
$y_0$ and $y_1$ are \emph{$C^{\infty}$ equivalent}
(or \emph{$C^0$ equivalent})
if for some open neighborhoods $U_i$ of $y_i$ in
$N_i$, there exist diffeomorphisms (respectively, homeomorphisms) 
$\tilde{\varphi} \co (f_0)^{-1}(U_0) \to (f_1)^{-1}(U_1)$ and 
$\varphi \co U_0 \to U_1$ with $\varphi(y_0) = y_1$
which make the following diagram commutative:
$$\bfig
\barrsquare<1500,400>[((f_0)^{-1}(U_0), (f_0)^{-1}(y_0))`
  ((f_1)^{-1}(U_1), (f_1)^{-1}(y_1))`
  (U_0, y_0)`
  (U_1, y_1);
  \tilde{\varphi}`f_0`f_1`\varphi]
  \efig$$
When the fibers over $y_0$ and $y_1$ are $C^\infty$ (or $C^{0}$) 
equivalent, we also say that the map germs $f_0 \co 
(M_0, (f_0)^{-1}(y_0)) \to (N_0, y_0)$ and $f_1 \co
(M_1, (f_1)^{-1}(y_1)) \to (N_1, y_1)$ are smoothly 
(or topologically) \emph{right-left equivalent}. 

When $y \in N$ is a regular value of a smooth map 
$f \co M \to N$ between smooth manifolds, we call $f^{-1}(y)$ 
a \emph{regular fiber}; otherwise, a \emph{singular fiber}.
\end{dfn}

\begin{dfn}\label{rem:codim}
Let $\mathfrak{F}$ be a $C^0$ equivalence class
of a fiber of a proper smooth map in the sense
of \fullref{dfn:equivalence}.
For a proper smooth map $f \co M \to N$ between
smooth manifolds, we denote by $\mathfrak{F}(f)$
the set consisting of those points of $N$
over which lies a fiber of type $\mathfrak{F}$.
It is known that if the smooth map $f$
is generic enough (for example if $f$
is a Thom map, see the book by Gibson, Wirthm\"uller, du Plessis and
Looijenga \cite{GWDL}), then
$\mathfrak{F}(f)$ is a union of strata\footnote{In the
case where $f$ is a Thom map, we consider the stratifications
of $M$ and $N$ with respect to which $f$ satisfies certain
regularity conditions. For details, see \cite[Chapter I, Section~3]{GWDL}.}
of $N$ and is a $C^0$ submanifold of $N$
of constant codimension
(for details, see \cite[Chapter~7]{Saeki04}).
Furthermore, this codimension $\kappa =
\kappa(\mathfrak{F})$ does not
depend on the choice of $f$ and we call it
the \emph{codimension} of $\mathfrak{F}$.
We also say that a fiber belonging to
$\mathfrak{F}$ is a \emph{codimension $\kappa$
fiber}. 
\end{dfn}

Let us introduce the following weaker relation
for (singular) fibers.

\begin{dfn}\label{dfn:modulo-regular}
Let $f_i \co (M_i, (f_i)^{-1}(y_i)) \to (N_i, y_i)$
be proper smooth map germs along fibers with
$n = \dim M_i$ and $p = \dim N_i$, $i = 0, 1$,
with $n \geq p$. We may assume that $N_i$ 
is the $p$--dimensional open disk $\Int D^p$
and that $y_i$ is its center $0$, $i = 0, 1$.
We say that the two fibers are $C^0$ (or $C^\infty$)
\emph{equivalent modulo regular fibers} if
there exist $(n-p)$--dimensional closed manifolds
$F_i$, $i = 0, 1$, such that
the disjoint union of $f_0$ and
the map germ $\pi_0 \co (F_0 \times \Int D^p,
F_0 \times \{0\}) \to (\Int D^p, 0)$
defined by the projection to the second factor
is $C^0$ (respectively, $C^\infty$) equivalent
to the disjoint union of $f_1$ and
the map germ $\pi_1 \co (F_1 \times \Int D^p,
F_1 \times \{0\}) \to (\Int D^p, 0)$
defined by the projection to the second factor.

Note that by the very definition,
any two regular fibers are
$C^\infty$ equivalent modulo regular
fibers to each other as long
as their dimensions of the source and the target
are the same.
\end{dfn}

For the $C^0$ equivalence modulo regular
fibers, we use the same notation as in
\fullref{rem:codim}. Then all the
assertions in \fullref{rem:codim}
hold for $C^0$ equivalence classes modulo
regular fibers as well.

The following definition is not so
important in this paper. However, in
order to compare it with \fullref{dfn:achiral},
we recall it. For details, refer to \cite{Saeki04}.

\begin{dfn}\label{dfn:co-ori}
Let $\mathfrak{F}$ be
a $C^0$ equivalence class of 
a fiber of a proper Thom map. Let us consider
arbitrary homeomorphisms 
$\tilde{\varphi}$ and $\varphi$ which make the diagram
$$\bfig
  \barrsquare<1500,400>[(f^{-1}(U_0), f^{-1}(y))`
    (f^{-1}(U_1),f^{-1}(y))`
    (U_0,y)`
    (U_1, y);
    \tilde{\varphi}`f`f`\varphi]
  \efig$$
commutative, 
where $f$ is a proper Thom map such that the fiber over $y$ belongs to 
$\mathfrak{F}$, and $U_i$ are open neighborhoods of $y$. 
Note that then we have 
$\varphi(\mathfrak{F}(f) \cap U_0)
= \mathfrak{F}(f) \cap U_1$. We say that 
$\mathfrak{F}$ is \emph{co-orientable} 
if $\varphi$ always preserves the
local orientation of the normal bundle of
$\mathfrak{F}(f)$ at $y$.

We also call any fiber belonging to a co-orientable $C^0$ equivalence 
class a \emph{co-orientable fiber}. 

In particular, 
if the codimension of $\mathfrak{F}$ coincides with the dimension 
of the target of $f$, then $\varphi$ above should preserve the 
local orientation of the target at $y$.

Note that if $\mathfrak{F}$ is co-orientable, then $\mathfrak{F}(f)$ 
has orientable normal bundle for every proper Thom map $f$.
\end{dfn}

The following definition plays an essential role
in this paper. Compare this with
\fullref{dfn:co-ori}.

\begin{dfn}\label{dfn:achiral}
Let $\mathfrak{F}$ be 
a $C^0$ equivalence class of a
fiber of a proper Thom map of an \emph{oriented}
manifold. We say that $\mathfrak{F}$ is
\emph{achiral} if there exist
homeomorphisms $\tilde{\varphi}$ and 
$\varphi$ which make the diagram
\begin{equation}
\bfig
\barrsquare<1500,400>[(f^{-1}(U_0), f^{-1}(y))`
    (f^{-1}(U_1), f^{-1}(y))`
    (U_0, y)`
    (U_1, y);
    \tilde{\varphi}`f`f`\varphi]
\efig
\label{eq:chiral}
\end{equation}
commutative such that the homeomorphism
$\tilde{\varphi}$ 
reverses the orientation and that the
homeomorphism 
\begin{equation}
\varphi|_{\mathfrak{F}(f) \cap U_0} \co
\mathfrak{F}(f) \cap U_0 \to \mathfrak{F}(f) \cap U_1
\label{eq:base}
\end{equation}
preserves the local orientation of $\mathfrak{F}(f)$ at $y$, 
where $f$ is a proper Thom map such that the fiber over $y$ belongs to 
$\mathfrak{F}$, and $U_i$ are open neighborhoods of $y$. 

Note that if the codimension of $\mathfrak{F}$ coincides with
the dimension of the target of $f$, then the condition about
the homeomorphism \eqref{eq:base} is redundant.
Note also that the above definition does not
depend on the choice of $f$ or $y$.

Moreover, we say that $\mathfrak{F}$
is \emph{chiral} if it is not achiral.

We also call any fiber belonging to a chiral (respectively, achiral) 
$C^0$ equivalence 
class a \emph{chiral fiber} (respectively, \emph{achiral fiber}). 
\end{dfn}

Furthermore, we have the following.

\begin{lem}
Suppose that the codimension of $\mathfrak{F}$
is strictly smaller than the dimension of the
target. Then
$\mathfrak{F}$ is achiral if and only if
there exist homeomorphisms $\tilde{\varphi}$
and $\varphi$
making the diagram \textup{\eqref{eq:chiral}}
commutative such that the homeomorphism $\tilde{\varphi}$
preserves the orientation and that the
homeomorphism \textup{\eqref{eq:base}} reverses the
orientation.
\end{lem}

\begin{proof}
Let $f \co (M, f^{-1}(y)) \to (N, y)$ be a representative
of $\mathfrak{F}$, which is a proper Thom map. 
Let us consider the Whitney stratifications
$\mathcal{M}$ and $\mathcal{N}$
of $M$ and $N$ respectively with respect to which $f$ satisfies
certain regularity conditions \cite[Chapter I, Section~3]{GWDL}. We may
assume that $y$ belongs to a top dimensional
stratum of $\mathfrak{F}(f)$ with
respect to $\mathcal{N}$. By our hypothesis,
the dimension $k$ 
of this stratum is strictly positive.
Let $\Delta$ be a
small open disk of codimension $k$
centered at $y$ in $N$ which intersects
with the stratum transversely at $y$.
Set $M' = f^{-1}(\Delta)$. Note that $f'
= f|_{M'} \co M' \to \Delta$ is a proper
Thom map.

By the second isotopy lemma (for example,
see \cite[Chapter II, Section~5]{GWDL}), we see that the
map germ
$$f \co (M, f^{-1}(y)) \to (N, y)$$
is $C^0$ equivalent to 
the map germ
$$f' \times \id_{\R^k} \co (M'
\times \R^k, {f'}^{-1}(y) \times 0)
\to (\Delta \times \R^k, y \times 0).$$ 
Since $k$
is positive, it is now easy to
construct orientation reversing homeomorphisms
$\tilde{\varphi}$
and $\varphi$
making the diagram \textup{\eqref{eq:chiral}}
commutative
such that the
homeomorphism \textup{\eqref{eq:base}} reverses the
orientation.
Then the lemma
follows immediately.
This completes the proof.
\end{proof}

We warn the reader that even if a fiber is
chiral, homeomorphisms
$\tilde{\varphi}$ and $\varphi$
making the diagram \eqref{eq:chiral} commutative
may not satisfy any of the following.
\begin{itemize}
\item[$(1)$] The homeomorphism $\tilde{\varphi}$
preserves the orientation and the homeomorphism \eqref{eq:base}
preserves the orientation.
\item[$(2)$] The homeomorphism $\tilde{\varphi}$
reverses the orientation and the 
homeomorphism \eqref{eq:base}
reverses the orientation.
\end{itemize}
This is because $f^{-1}(U_i)$ may
not be connected. 

For example, a regular fiber is achiral
if and only if the fiber manifold admits
an orientation reversing homeomorphism.
The disjoint union of an achiral fiber
and an achiral regular fiber
is clearly achiral.
The disjoint union of a chiral fiber
and an achiral regular fiber is always chiral.

In what follows, we consider only those maps of
codimension $-1$ so that a regular
fiber is always of dimension $1$. 
Note that every compact $1$--dimensional
manifold admits an orientation reversing homeomorphism.
Therefore, for two fibers which are
$C^0$ equivalent modulo regular fibers,
one is chiral if and only if so is the other.
Therefore, 
we can speak of a chiral (or achiral)
$C^0$ equivalence class modulo regular
fibers as well.

\section{Singular fibers of stable
maps of $4$--manifolds into $3$--man\-i\-folds}\label{section3}

In this section, we consider proper $C^\infty$ stable maps
of orientable $4$--manifolds into $3$--manifolds and
recall the classification of their singular fibers
obtained in \cite{Saeki04}.

Let $M$ and $N$ be manifolds.
We say that a smooth map $f \co M \to N$ is \emph{$C^\infty$ stable}
(or \emph{stable} for short) if the $\mathcal{A}$--orbit of $f$ is open
in the mapping space $C^{\infty}(M,N)$ with respect to 
the Whitney $C^{\infty}$--topology.
Here, the $\mathcal{A}$--orbit of $f \in C^{\infty}(M,N)$ means
the following.
Let $\Diff{M}$ (or $\Diff{N}$)
denote the group of all diffeomorphisms of
the manifold $M$ (respectively, $N$). 
Then $\Diff{M} \times \Diff{N}$ acts on $C^{\infty}(M,N)$ by
$(\Phi, \Psi) \cdot f = \Psi \circ f \circ {\Phi}^{-1}$ for
$(\Phi, \Psi) \in \Diff{M} \times \Diff{N}$ and 
$f \in C^{\infty}(M,N)$.
Then the $\mathcal{A}$--orbit of $f \in C^{\infty}(M,N)$ means 
the orbit through $f$ with respect to this action.
Note that a proper $C^\infty$ stable map is always
a Thom map.

Since $(4, 3)$ is a nice dimension pair
in the sense of Mather \cite{MatherVI},
if $\dim M = 4$ and $\dim N = 3$, then
the set of all $C^\infty$ stable maps is
open and dense in $C^\infty(M, N)$ as long
as $M$ is compact. In particular, every
smooth map $M \to N$ can be approximated
arbitrarily well by a $C^\infty$ stable map.
This shows the abundance of such stable maps.

The following characterization of proper
$C^\infty$ stable maps of $4$--manifolds into
$3$--manifolds is well-known (for example, see
\cite{Saeki04}).

\begin{prop}\label{prop:stable}
A proper smooth map $f \co M \to N$ of a $4$--manifold $M$ into
a $3$--manifold $N$ is $C^\infty$ stable
if and only if the following conditions are
satisfied.

\textup{(i)}\qua
For every $q \in M$, there exist
local coordinates $(x, y, z, w)$ and $(X, Y, Z)$
around $q \in M$ and $f(q) \in N$ respectively
such that one of the following holds:
$$(X {\circ} f, Y {\circ} f, Z {\circ} f) =
\begin{cases}
(x, y, z) & \text{$q$ a regular point} \\
(x, y, z^2{+}w^2) & \text{$q$ a definite fold point} \\
(x, y, z^2{-}w^2) & \text{$q$ an indefinite fold point} \\
(x, y, z^3{+}xz{-}w^2) & \text{$q$ a cusp point} \\
(x, y, z^4{+}xz^2{+}yz{+}w^2) & \text{$q$ a definite swallowtail point} \\
(x, y, z^4{+}xz^2{+}yz{-}w^2) & \text{$q$ an indefinite swallowtail point} 
\end{cases}$$
\textup{(ii)}\qua
Set $S(f) = \{q \in M : \rank df_q < 3\}$, which
is a regular closed $2$--dimensional submanifold of $M$
under the above condition \textup{(i)}. Then, for every $r \in f(S(f))$,
$f^{-1}(r) \cap S(f)$ consists of at most three points and
the multi-germ 
$$(f|_{S(f)}, f^{-1}(r) \cap S(f))$$ 
is smoothly right-left
equivalent to one of the six multi-germs as described in
\fullref{fig3}: $(1)$ represents a single immersion
germ which corresponds to a fold point,
$(2)$ and $(4)$ represent normal crossings of two and three
immersion germs, respectively, each
of which corresponds to a fold point, 
$(3)$ corresponds to a cusp point, $(5)$ represents
a transverse crossing of a cuspidal edge 
\index{cuspidal edge}
as in $(3)$
and an immersion germ corresponding to a fold
point, and $(6)$ corresponds to a swallowtail point.
\end{prop}

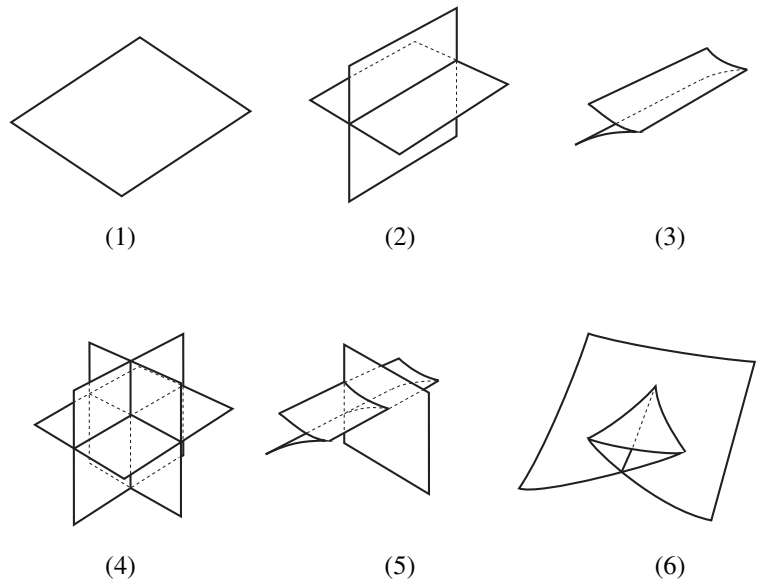
\begin{figure}[ht!]
\labellist\small
\pinlabel {$(1)$} at 124 571
\pinlabel {$(2)$} at 394 571
\pinlabel {$(3)$} at 655 571
\pinlabel {$(4)$} at 124 253
\pinlabel {$(5)$} at 394 253
\pinlabel {$(6)$} at 655 253
\endlabellist
\centerline{\includegraphics[scale=0.4]{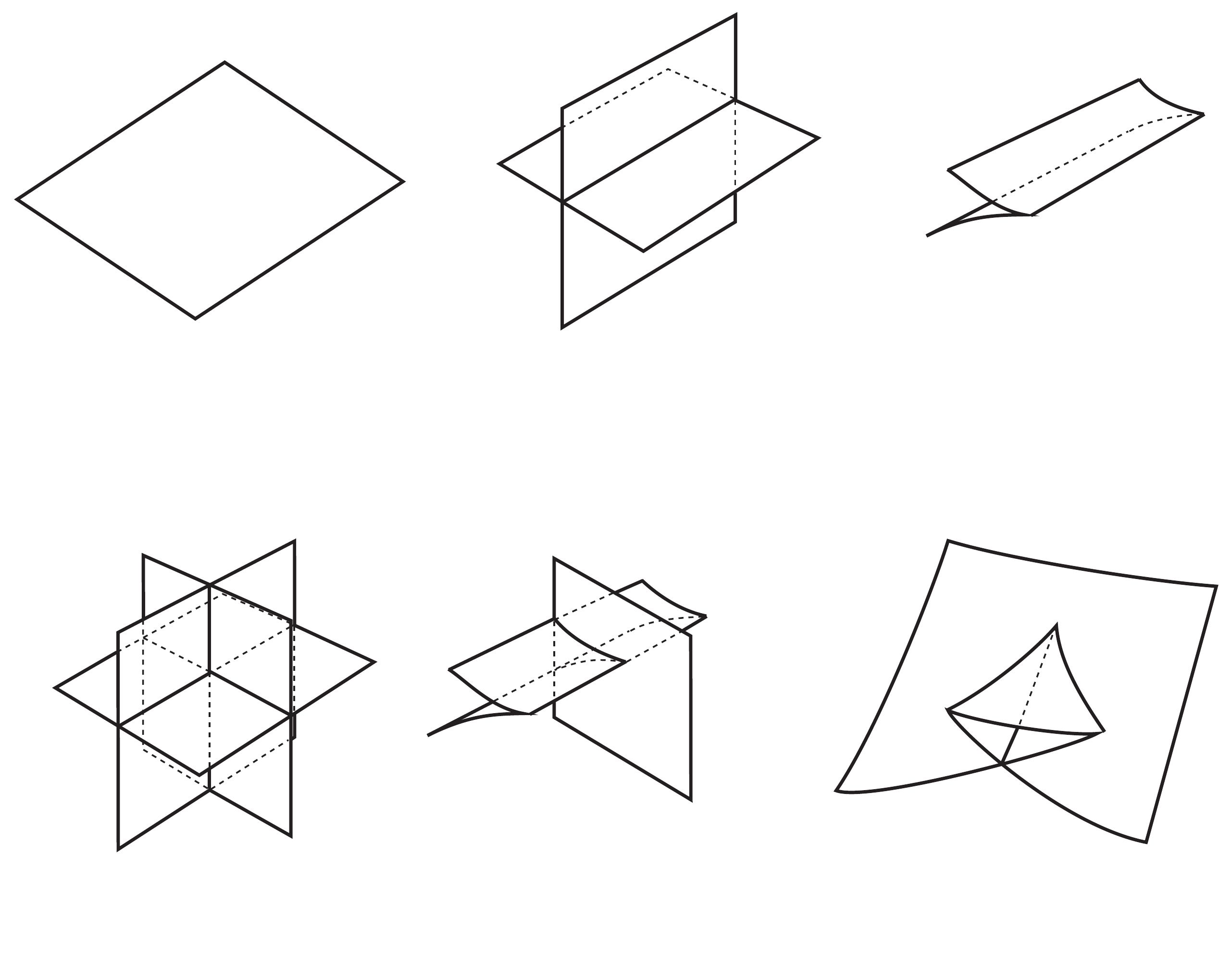}}
\caption{Multi-germs of $f|_{S(f)}$}
\label{fig3}
\end{figure}

In the following, we assume that the $4$--manifold
$M$ is orientable. Using \fullref{prop:stable},
the first named author obtained the following classification
of singular fibers \cite{Saeki04}.
\eject

\begin{thm}\label{thm:4classification}
Let $f \co M \to N$ be a proper $C^\infty$ stable map
of an orientable $4$--manifold $M$ into
a $3$--manifold $N$. Then, every singular fiber of $f$
is $C^\infty$ (and hence $C^0$)
equivalent modulo regular
fibers to one of the fibers as in
\fullref{fig4}. 
Furthermore, no two fibers appearing in 
the list are $C^0$ equivalent modulo
regular fibers.
\end{thm}

\begin{figure}[htp!]
\centering
\labellist\small
\pinlabel {$\kappa = 1$} [l] at -213 751
\pinlabel {$\mathrm{I}^0$} [l] at -90 751
\pinlabel {$\mathrm{I}^1$} [l] at 110 751
\pinlabel {$\kappa = 2$} [l] at -213 655
\pinlabel {$\mathrm{II}^{0,0}$} [l] at -90 655
\pinlabel {$\mathrm{II}^{0,1}$} [l] at 110 655
\pinlabel {$\mathrm{II}^{1,1}$} [l] at 360 655
\pinlabel {$\mathrm{II}^2$} [l] at -90 540
\pinlabel {$\mathrm{II}^3$} [l] at 110 540
\pinlabel {$\mathrm{II}^a$} [l] at 360 540
\pinlabel {$\kappa = 3$} [l] at -213 387
\pinlabel {$\mathrm{III}^{0,0,0}$} [l] at -100 390
\pinlabel {$\mathrm{III}^{0,0,1}$} [l] at 130 390
\pinlabel {$\mathrm{III}^{0,1,1}$} [l] at 390 390
\pinlabel {$\mathrm{III}^{1,1,1}$} [l] at -100 240
\pinlabel {$\mathrm{III}^{0,2}$} [l] at 180 240
\pinlabel {$\mathrm{III}^{0,3}$} [l] at 420 240
\pinlabel {$\mathrm{III}^{1,2}$} [l] at -100 70
\pinlabel {$\mathrm{III}^{1,3}$} [l] at 160 70
\pinlabel {$\mathrm{III}^4$} [l] at 450 70
\pinlabel {$\mathrm{III}^5$} [l] at -100 -110
\pinlabel {$\mathrm{III}^6$} [l] at 125 -110
\pinlabel {$\mathrm{III}^7$} [l] at 370 -110
\pinlabel {$\mathrm{III}^8$} [l] at -100 -260
\pinlabel {$\mathrm{III}^{0,a}$} [l] at 150 -260
\pinlabel {$\mathrm{III}^{1,a}$} [l] at 410 -260
\pinlabel {$\mathrm{III}^b$} [l] at -100 -440
\pinlabel {$\mathrm{III}^c$} [l] at 150 -440
\pinlabel {$\mathrm{III}^d$} [l] at 320 -440
\pinlabel {$\mathrm{III}^e$} [l] at -100 -600
\endlabellist
\includegraphics[width=10cm]{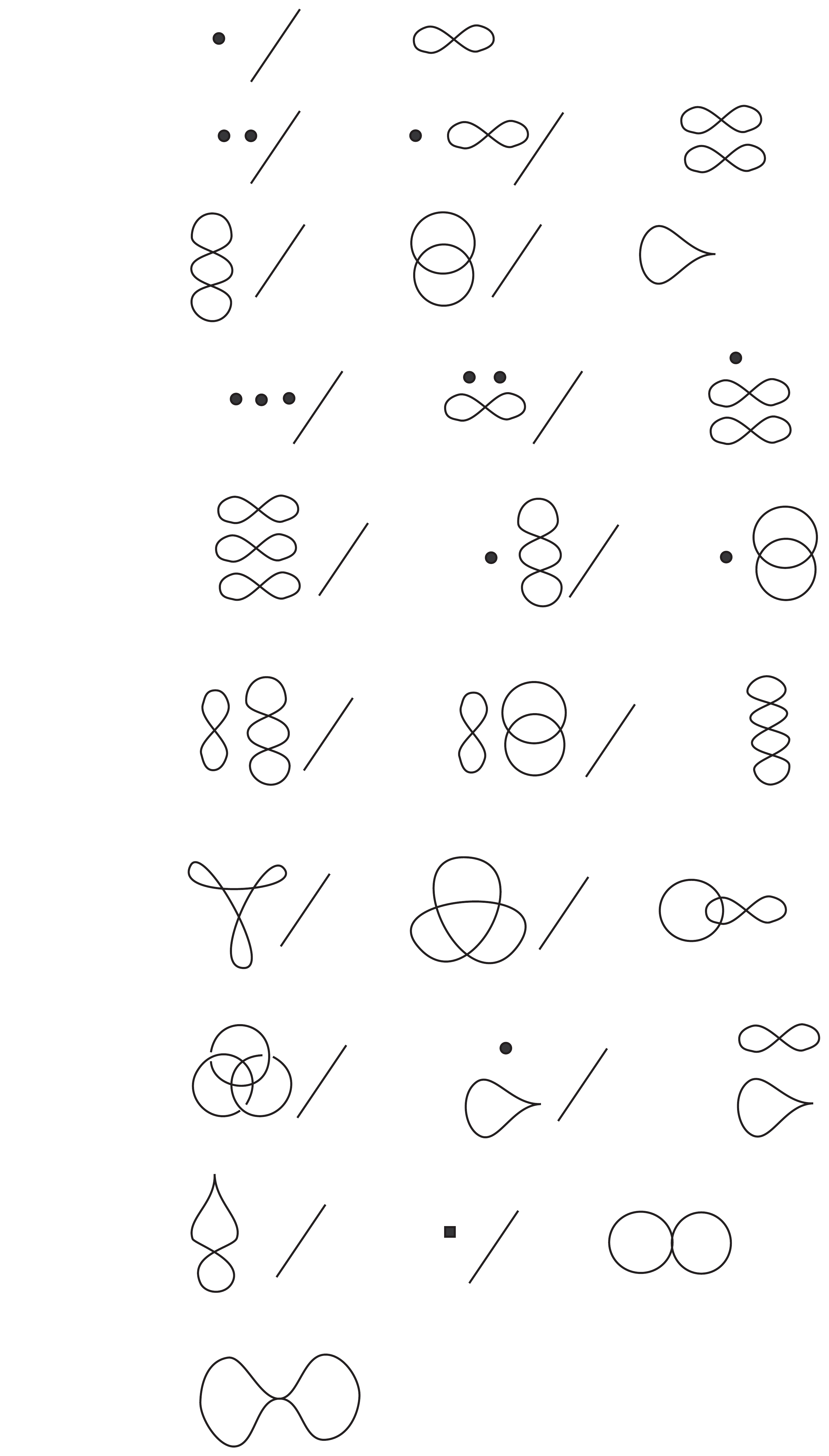}
\caption{List of singular fibers of proper $C^\infty$ stable maps of
orientable $4$--manifolds into $3$--manifolds}
\index{singular fiber!list}
\label{fig4}
\end{figure}
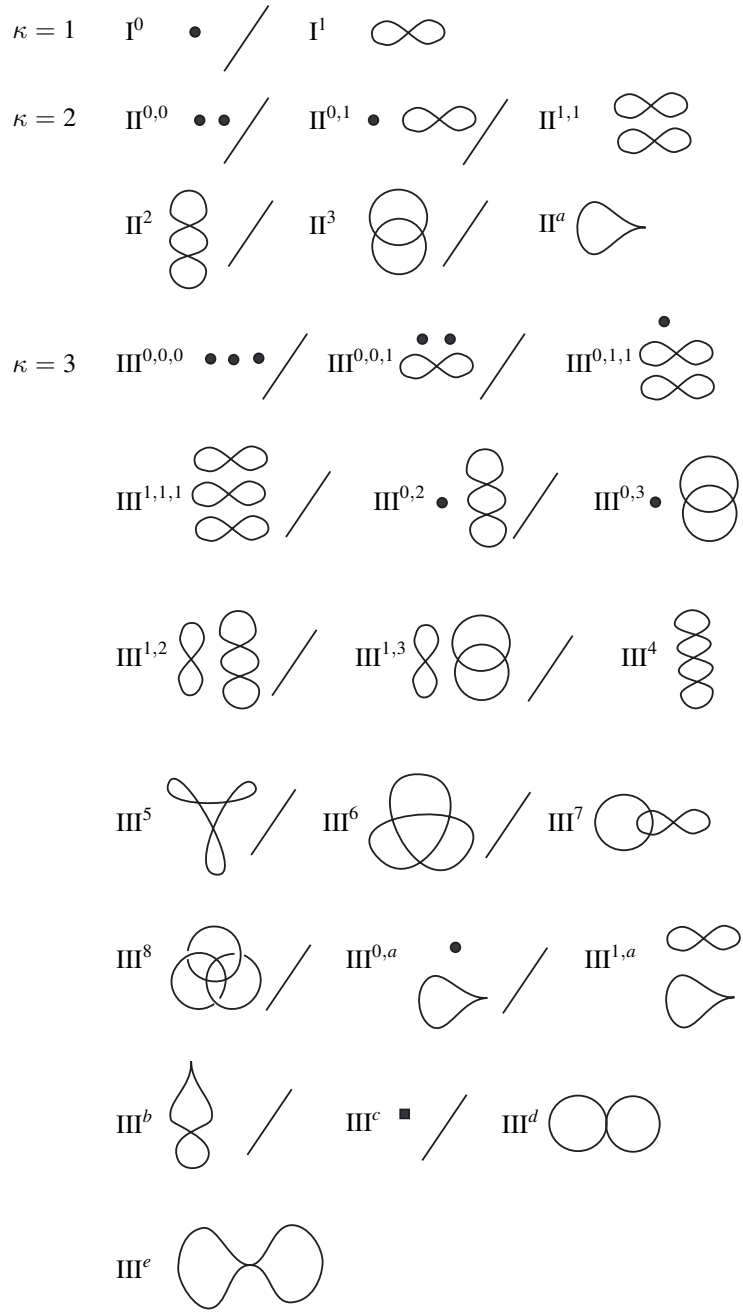

\begin{rem}\label{rmk:name}
In \fullref{fig4},
$\kappa$ denotes the codimension 
of the relevant singular fiber
in the sense of \fullref{rem:codim}.
Furthermore, $\mathrm{I}^*, \mathrm{II}^*$ and
$\mathrm{III}^*$ mean the names of the corresponding
singular fibers, and ``$/$" is used only for
separating the figures. Note that we have named the
fibers so that each connected fiber has its
own digit or letter, and a disconnected fiber
has the name consisting of the digits or letters
of its connected components. Hence,
the number of digits or letters in the superscript
coincides with the number of connected components
that contain singular points.
\end{rem}

\begin{rem}
For proper $C^\infty$ stable maps
of $3$--manifolds into the plane,
a similar classification of singular fibers
was obtained by Kushner, Levine and Porto
\cite{KLP,Levine1}, although they
did not mention explicitly the equivalence
relation for their classification.
Their classification was in fact
based on the ``diffeomorphism
modulo regular fibers".
\end{rem}

\begin{rem}
For proper $C^\infty$ stable maps of
general (possibly nonorientable)
$4$--manifolds
into $3$--manifolds, a similar classification
of singular fibers was obtained by the
second named author in \cite{YT1,YT2}.
\end{rem}

\section{Singular fibers of stable maps 
of $5$--manifolds into $4$--man\-i\-folds}\label{section4}

In this section we give a characterization of 
$C^\infty$ stable maps of $5$--manifolds into $4$--manifolds
and present the classification of singular fibers
of such maps.

First note that since $(5, 4)$ is a nice
dimension pair in the sense of Mather
\cite{MatherVI}, for a $5$--manifold $M$ and
a $4$--manifold $N$, the set of all $C^\infty$
stable maps is open and dense in the mapping
space $C^\infty(M, N)$, as long as $M$
is compact. 

By using standard methods in singularity theory
(for example, see the book \cite{GG} by Golubitsky and Guillemin),
together with a result of Ando \cite{Ando}, we can prove
the following characterization of stable maps
of $5$--manifolds into $4$--manifolds.

\begin{prop}\label{prop:5chara}
A proper smooth map $f \co M \to N$ of a $5$--manifold $M$ into
a $4$--manifold $N$ is $C^\infty$ stable
if and only if the following conditions are
satisfied.

\textup{(i)}
For every $q \in M$, there exist
local coordinates $(a, b, c, x, y)$ and $(X, Y, Z, W)$
around $q \in M$ and $f(q) \in N$ respectively
such that one of the following holds:
\begin{multline*}
(X {\circ} f,Y {\circ} f,Z {\circ} f,W {\circ} f) = \\
\begin{cases}
  (a,b,c,x)                       & \text{$q$ a regular point}\\
  (a,b,c,x^2{+}y^2)               & \text{$q$ a definite fold point}\\
  (a,b,c,x^2{-}y^2)               & \text{$q$ an indefinite fold point}\\
  (a,b,c,x^3{+}ax{-}y^2)          & \text{$q$ a cusp point}\\
  (a,b,c,x^4{+}ax^2{+}bx{+}y^2)   & \text{$q$ a definite swallowtail point} \\
  (a,b,c,x^4{+}ax^2{+}bx{-}y^2)   &
    \text{$q$ an indefinite swallowtail point} \\
  (a,b,c,x^5{+}ax^3{+}bx^2{+}cx{-}y^2) &
    \text{$q$ a butterfly point}\\
  (a,b,c,3x^2y{+}y^3{+}a(x^2{+}y^2){+}bx{+}cy) &
    \text{$q$ a definite $D_4$ point}\\
  (a,b,c,3x^2y{-}y^3{+}a(x^2{+}y^2){+}bx{+}cy) &
    \text{$q$ an indefinite $D_4$ point} 
\end{cases}
\end{multline*}
\textup{(ii)}
Set $S(f) = \{q \in M\,:\, \rank df_q < 4\}$, which
is a regular closed $3$--dimensional submanifold of $M$
under the above condition \textup{(i)}. 
Then, for every $r \in f(S(f))$,
$f^{-1}(r) \cap S(f)$ consists of at most four points and
the multi-germ 
$$(f|_{S(f)}, f^{-1}(r) \cap S(f))$$ 
is smoothly right-left
equivalent to one of the thirteen multi-germs
as follows:
\begin{enumerate}
\item A single immersion germ which corresponds to a fold point 
\item A normal crossing of two immersion germs, each of which
  corresponds to a fold point
\item A cuspidal edge which corresponds to a single cusp point
\item A normal crossing of three immersion germs, each of which
  corresponds to a fold point
\item A transverse crossing of a cuspidal edge
  and an immersion germ corresponding to a fold point
\item A map germ corresponding to a swallowtail point
\item A normal crossing of four immersion germs, each of which
  corresponds to a fold point
\item A transverse crossing of a cuspidal edge and a normal crossing 
  of two immersion germs which correspond to fold points 
\item A transverse crossing of two cuspidal edges 
\item A transverse crossing of a swallowtail germ and an immersion germ
  corresponding to a fold point 
\item A map germ corresponding to a butterfly point 
\item A map germ corresponding to a definite $D_4$ point 
\item A map germ corresponding to an indefinite $D_4$ point
\end{enumerate}
\end{prop}

We call a $D_4$ point 
a $\Sigma^{2,2,0}$ point as well.

\begin{rem}\label{rem:symmetry}
The normal forms for $D_4$ points
are slightly different from the usual ones
(see, for example, the article \cite{Ando} by Ando).
We have chosen them so that 
at an indefinite $D_4$ point, $f$ can be
represented as
$$(a, \eta, \zeta) \mapsto (a, \eta, \Im(\zeta^3) + 
\Re(\bar{\eta}\zeta) + a|\zeta|^2)$$
by using complex numbers,
where $i = \sqrt{-1}$, $\eta = b + ic$, $\zeta = x + iy$,
$\Im$ means the imaginary part, and $\Re$ means
the real part.

Set $\tau = \exp{(2 \pi i/3)}$. Then with respect
to the chosen coordinates, we have
$$f \circ \tilde{\varphi}_\tau = \varphi_\tau \circ f,$$
where $\tilde{\varphi}_\tau$ and $\varphi_\tau$ are
orientation preserving diffeomorphisms
defined by 
\begin{eqnarray*}
\tilde{\varphi}_\tau(a, \eta, \zeta)
& = & (a, \tau \eta, \tau \zeta), \quad \text{ and} \\
\varphi_\tau(X, Y + iZ, W) & = & (X, \tau (Y + iZ), W)
\end{eqnarray*}
respectively.
This shows that an indefinite $D_4$ point
(or a local fiber through an indefinite $D_4$
point)
has a (orientation preserving) symmetry of order $3$. 

Set $\tau' = \exp{(\pi i/3)}$ so that
we have ${\tau'}^2 = \tau$. Then we have
$$f \circ \tilde{\varphi}_{\tau'} = \varphi_{\tau'} \circ f,$$
where $\tilde{\varphi}_{\tau'}$ and $\varphi_{\tau'}$ are
diffeomorphisms
defined by 
\begin{eqnarray*}
\tilde{\varphi}_{\tau'}(a, \eta, \zeta)
& = & (-a, -\tau' \eta, \tau' \zeta), \quad \text{ and} \\
\varphi_\tau(X, Y + iZ, W) & = & (-X, -\tau' (Y + iZ), -W)
\end{eqnarray*}
respectively. Note that $\tilde{\varphi}_{\tau'}$
is orientation reversing while $\varphi_{\tau'}$
is orientation preserving. This shows
that an indefinite $D_4$ point
(or a local fiber through an indefinite $D_4$ point)
has a symmetry of order $6$ and that the
generator reverses the ``local orientation"
of the fiber. In fact,
we have $\tilde{\varphi}_{\tau} = \tilde{\varphi}_{\tau'}^2$
and $\varphi_\tau = \varphi_{\tau'}^2$.
\end{rem}

Let us recall the following definition
(for details, see \cite[Chapter~8]{Saeki04}).

\begin{dfn}\label{dfn:suspension}
Let $f \co M \to N$ be a proper smooth map between
manifolds. Then we call the proper smooth map
$$f \times \id_{\R} \co M \times \R \to N \times \R$$
the \emph{suspension} of $f$.
Furthermore, to the fiber of $f$ over a point $y
\in N$, we can associate the fiber of
$f \times \id_{\R}$
over $y \times \{0\}$. We say that the latter fiber is obtained
from the original fiber by the \emph{suspension}.
Note that a fiber and its suspension are
diffeomorphic to each other in the sense
of \fullref{dfn:diffeo}.
\end{dfn}

Note that the map germs (1)--(6) in 
\fullref{prop:5chara}
correspond to the suspensions of the map germs
in \fullref{fig3}. The map germs (11)--(13)
are as described in Figures~\ref{local1}--\ref{local3}
respectively,
where in order to draw $3$--dimensional objects
in a $4$--dimensional space, we have depicted
three ``sections" by $3$--dimensional spaces
for each object. 

\begin{figure}[ht!]
  \centering
  \includegraphics[width=10cm]{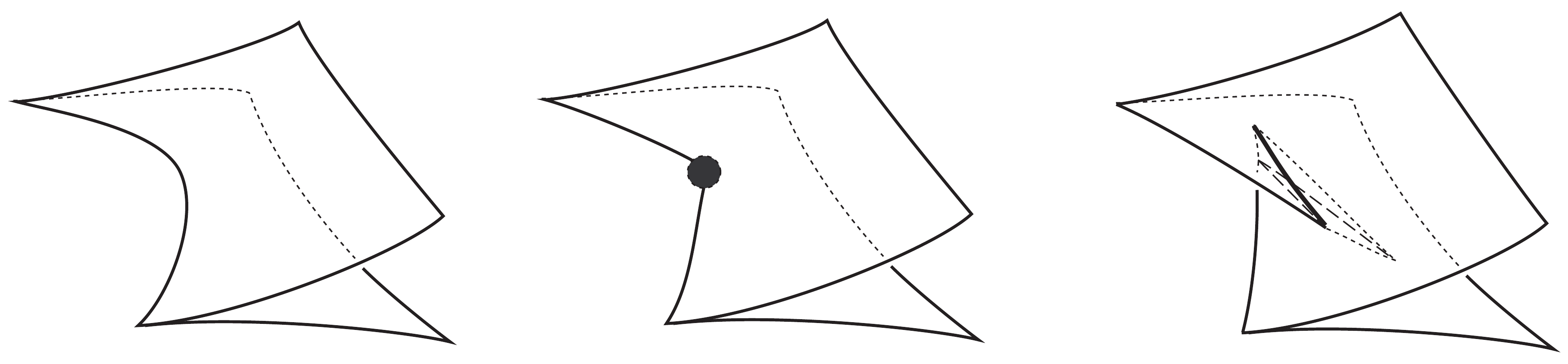}
  \caption{Map germ corresponding to a butterfly point}
  \label{local1}
\end{figure}

\begin{figure}[ht!]
  \centering
  \includegraphics[width=10cm]{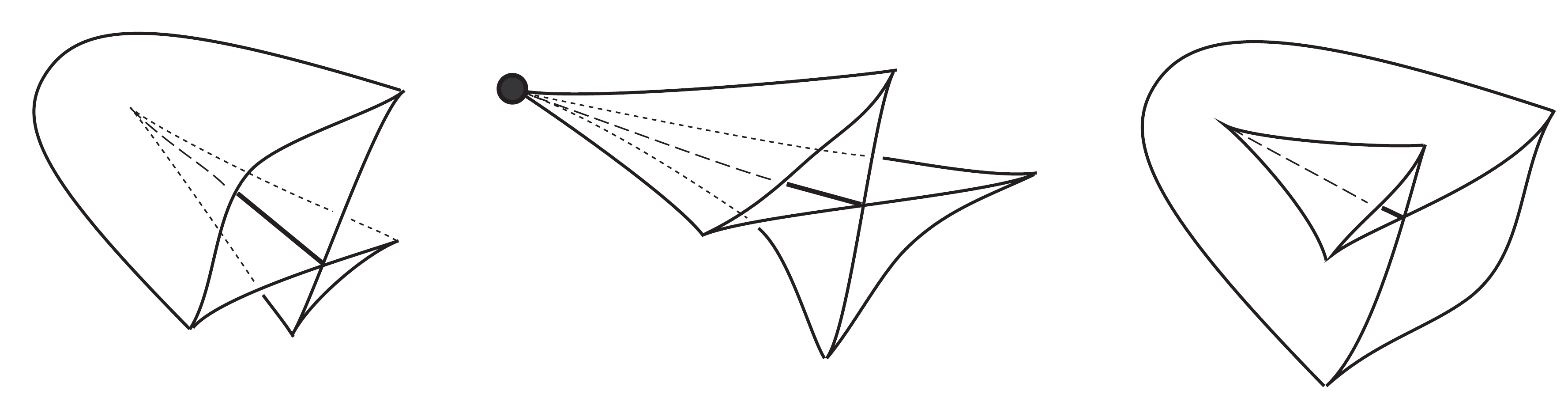}
  \caption{Map germ corresponding to a definite $D_4$ point}
  \label{local2}
\end{figure}

\begin{figure}[ht!]
  \centering
  \includegraphics[width=10cm]{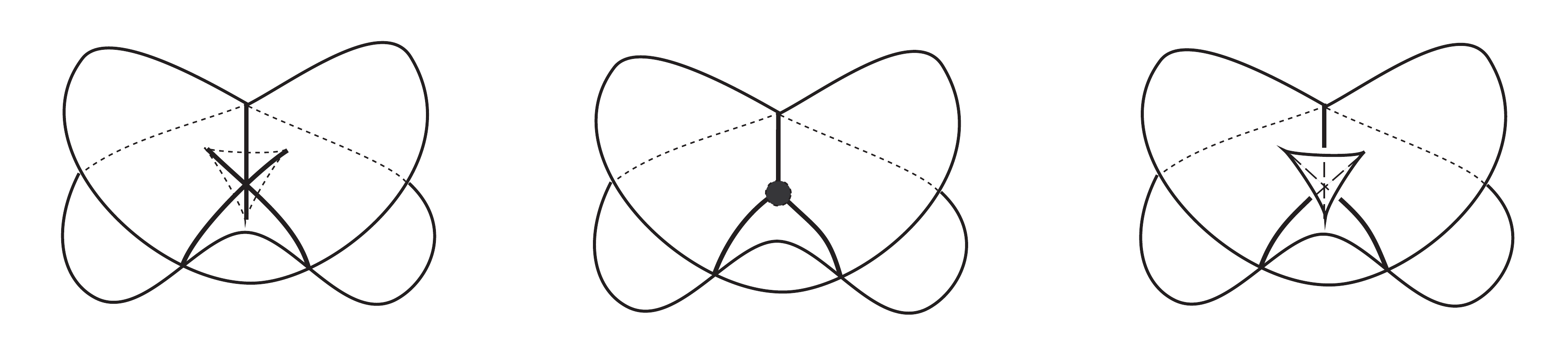}
  \caption{Map germ corresponding to an indefinite $D_4$ point}
  \label{local3}
\end{figure}

Let $q$ be a singular point of a proper $C^\infty$
stable map $f \co M \to N$
of a $5$--manifold $M$ into a $4$--manifold $N$.
Then, using the above local normal forms, we can easily
describe the diffeomorphism type of
a neighborhood of $q$ in $f^{-1}(f(q))$.
More precisely, we easily get the following local
characterizations of singular fibers.

\begin{lem}
Every singular point $q$ of a proper $C^\infty$
stable map $f\co M \to N$ of a $5$--manifold $M$ into 
a $4$--manifold $N$ has one of the following 
neighborhoods in its corresponding singular fiber\/ 
(see \fullref{local4}):
\begin{enumerate}
\item isolated point diffeomorphic to 
$\{ (x,y) \in \R^2\,:\, x^2+y^2=0 \}$, if $q$ is a definite fold point,
\item union of two transverse arcs diffeomorphic to 
$\{ (x,y) \in \R^2\,:\, x^2-y^2=0 \}$, if $q$ is an indefinite fold point,
\item $(2,3)$--cuspidal arc diffeomorphic to 
$\{ (x,y) \in \R^2 \,:\, x^3-y^2=0 \}$, if $q$ is a cusp point,
\item isolated point diffeomorphic to 
$\{ (x,y) \in \R^2 \,:\, x^4+y^2=0 \}$, if $q$ is a definite swallowtail point,
\item union of two tangent arcs diffeomorphic to 
$\{ (x,y) \in \R^2 \,:\, x^4-y^2=0 \}$, if $q$ is an indefinite swallowtail point,
\item $(2,5)$--cuspidal arc diffeomorphic to 
$\{ (x,y) \in \R^2 \,:\, x^5-y^2=0 \}$, if $q$ is a butterfly point,
\item non-disjoint union of an arc and a point diffeomorphic to 
$\{ (x,y) \in \R^2 \,:\,  3x^2y+y^3=0 \}$, if $q$ is a definite 
$D_4$ point,
\item union of three arcs meeting at a point with distinct
tangents diffeomorphic to 
$\{ (x,y) \in \R^2 \,:\, 3x^2y-y^3=0\}$, if $q$ is an indefinite 
$D_4$ point.
\end{enumerate}
\end{lem}

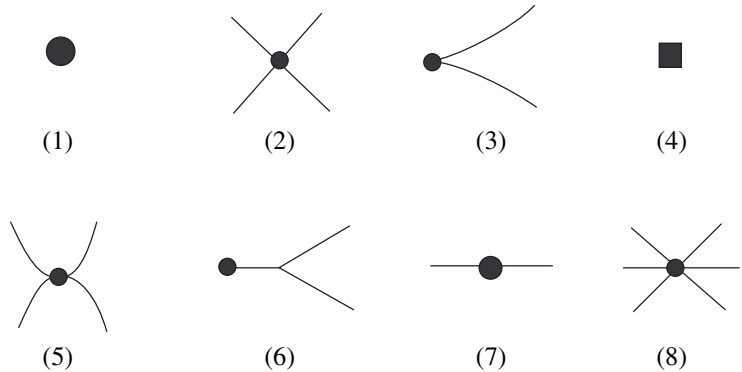
\begin{figure}[ht!]
\labellist\small
\pinlabel {$(1)$} at 44 652
\pinlabel {$(2)$} at 216 652
\pinlabel {$(3)$} at 380 652
\pinlabel {$(4)$} at 520 652
\pinlabel {$(5)$} at 44 484
\pinlabel {$(6)$} at 216 484
\pinlabel {$(7)$} at 380 484
\pinlabel {$(8)$} at 520 484
\endlabellist
\includegraphics[width=10cm]{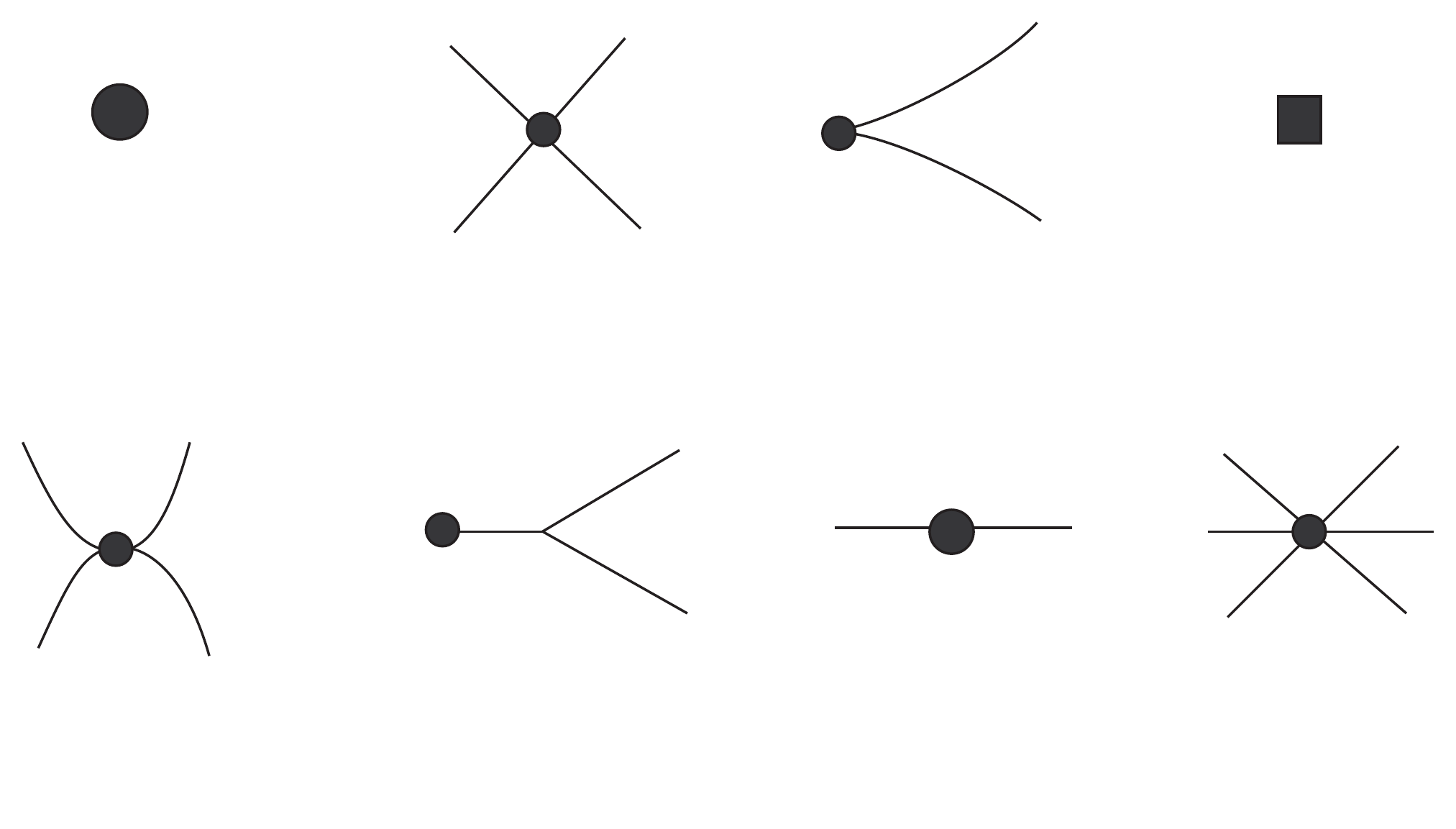}
\caption{Neighborhood of a singular point in a singular fiber}
\label{local4}
\end{figure}

We note that in \fullref{local4}, both the black dot 
(1) and the black square (4)
represent an isolated point;
however, we use distinct symbols in order to distinguish them.
We also use the symbols as in \fullref{local4} (3) and
(6) in order to distinguish a $(2, 3)$--cusp from a
$(2, 5)$--cusp. Furthermore,
we put a dot on the arc as in \fullref{local4} (7)
in order to distinguish it from a regular fiber.

Then by an argument similar to that in
\cite[Chapter~3]{Saeki04}, we can prove the
following, whose proof is left to the reader.

\begin{thm}\label{thm:5classification}
Let $f \co M \to N$ be a proper $C^\infty$ stable map
of an orientable $5$--manifold $M$ into
a $4$--manifold $N$. Then, every singular fiber of $f$
is $C^0$ equivalent modulo regular fibers to 
one of the fibers as follows:
\begin{enumerate}
\item The suspension of a fiber appearing in
\fullref{thm:4classification}
\item One of the disconnected fibers\\
$\mathrm{IV}^{0,0,0,0}$, $\mathrm{IV}^{0,0,0,1}$, 
$\mathrm{IV}^{0,0,1,1}$, $\mathrm{IV}^{0,1,1,1}$, 
$\mathrm{IV}^{1,1,1,1}$, $\mathrm{IV}^{0,0,2}$, 
$\mathrm{IV}^{0,1,2}$, $\mathrm{IV}^{1,1,2}$, 
$\mathrm{IV}^{0,0,3}$, $\mathrm{IV}^{0,1,3}$, 
$\mathrm{IV}^{1,1,3}$, $\mathrm{IV}^{0,4}$, 
$\mathrm{IV}^{0,5}$, $\mathrm{IV}^{0,6}$, 
$\mathrm{IV}^{0,7}$, $\mathrm{IV}^{0,8}$, 
$\mathrm{IV}^{1,4}$, 
$\mathrm{IV}^{1,5}$, $\mathrm{IV}^{1,6}$, 
$\mathrm{IV}^{1,7}$, $\mathrm{IV}^{1,8}$, 
$\mathrm{IV}^{2,2}$, 
$\mathrm{IV}^{2,3}$, $\mathrm{IV}^{3,3}$, 
$\mathrm{IV}^{0,0,a}$, $\mathrm{IV}^{0,1,a}$, 
$\mathrm{IV}^{1,1,a}$, 
$\mathrm{IV}^{0,b}$, $\mathrm{IV}^{1,b}$, 
$\mathrm{IV}^{2,a}$, $\mathrm{IV}^{3,a}$, 
$\mathrm{IV}^{a,a}$, 
$\mathrm{IV}^{0,c}$, $\mathrm{IV}^{0,d}$, 
$\mathrm{IV}^{0,e}$, $\mathrm{IV}^{1,c}$, 
$\mathrm{IV}^{1,d}$, 
$\mathrm{IV}^{1,e}$
\item One of the connected fibers
depicted in \fullref{fig5}
\end{enumerate}
Furthermore, no two fibers appearing in 
the list $(1)$--$(3)$ above
are $C^0$ equivalent 
modulo regular fibers.
\end{thm}

\begin{figure}[htp!]
\centering
\labellist\small
\pinlabel {$\mathrm{IV}^9$} at -250 522
\pinlabel {$\mathrm{IV}^{10}$} at -35 522
\pinlabel {$\mathrm{IV}^{11}$} at 190 522
\pinlabel {$\mathrm{IV}^{12}$} at 405 522
\pinlabel {$\mathrm{IV}^{13}$} at 600 522
\pinlabel {$\mathrm{IV}^{14}$} at 870 522
\pinlabel {$\mathrm{IV}^{15}$} at -225 135
\pinlabel {$\mathrm{IV}^{16}$} at -10 135
\pinlabel {$\mathrm{IV}^{17}$} at 205 135
\pinlabel {$\mathrm{IV}^{18}$} at 410 135
\pinlabel {$\mathrm{IV}^{19}$} at 640 135
\pinlabel {$\mathrm{IV}^{20}$} at 900 135
\pinlabel {$\mathrm{IV}^{21}$} at -245 -245
\pinlabel {$\mathrm{IV}^{22}$} at 15 -245
\pinlabel {$\mathrm{IV}^f$} at 230 -245
\pinlabel {$\mathrm{IV}^g$} at 440 -245
\pinlabel {$\mathrm{IV}^h$} at 685 -245
\pinlabel {$\mathrm{IV}^i$} at 925 -245
\pinlabel {$\mathrm{IV}^j$} at -205 -620
\pinlabel {$\mathrm{IV}^k$} at 0 -620
\pinlabel {$\mathrm{IV}^l$} at 240 -620
\pinlabel {$\mathrm{IV}^m$} at 500 -620
\pinlabel {$\mathrm{IV}^n$} at 735 -620
\pinlabel {$\mathrm{IV}^o$} at 960 -620
\pinlabel {$\mathrm{IV}^p$} at -200 -970
\pinlabel {$\mathrm{IV}^q$} at 75 -970
\endlabellist
\includegraphics[scale=0.27]{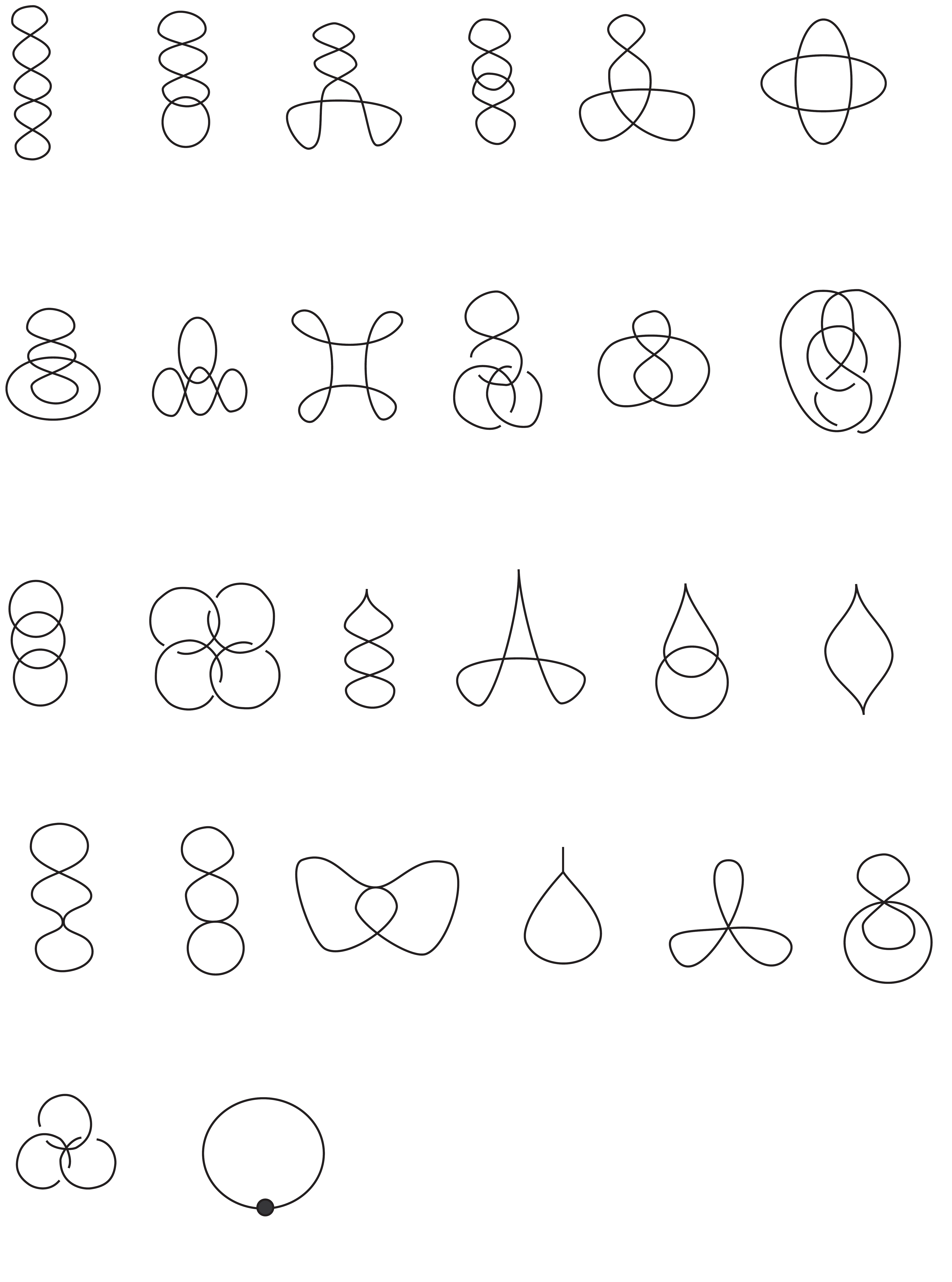}
\caption{List of codimension $4$ connected singular fibers of proper
$C^\infty$ stable maps of orientable $5$--manifolds into $4$--manifolds}
\label{fig5}
\end{figure}
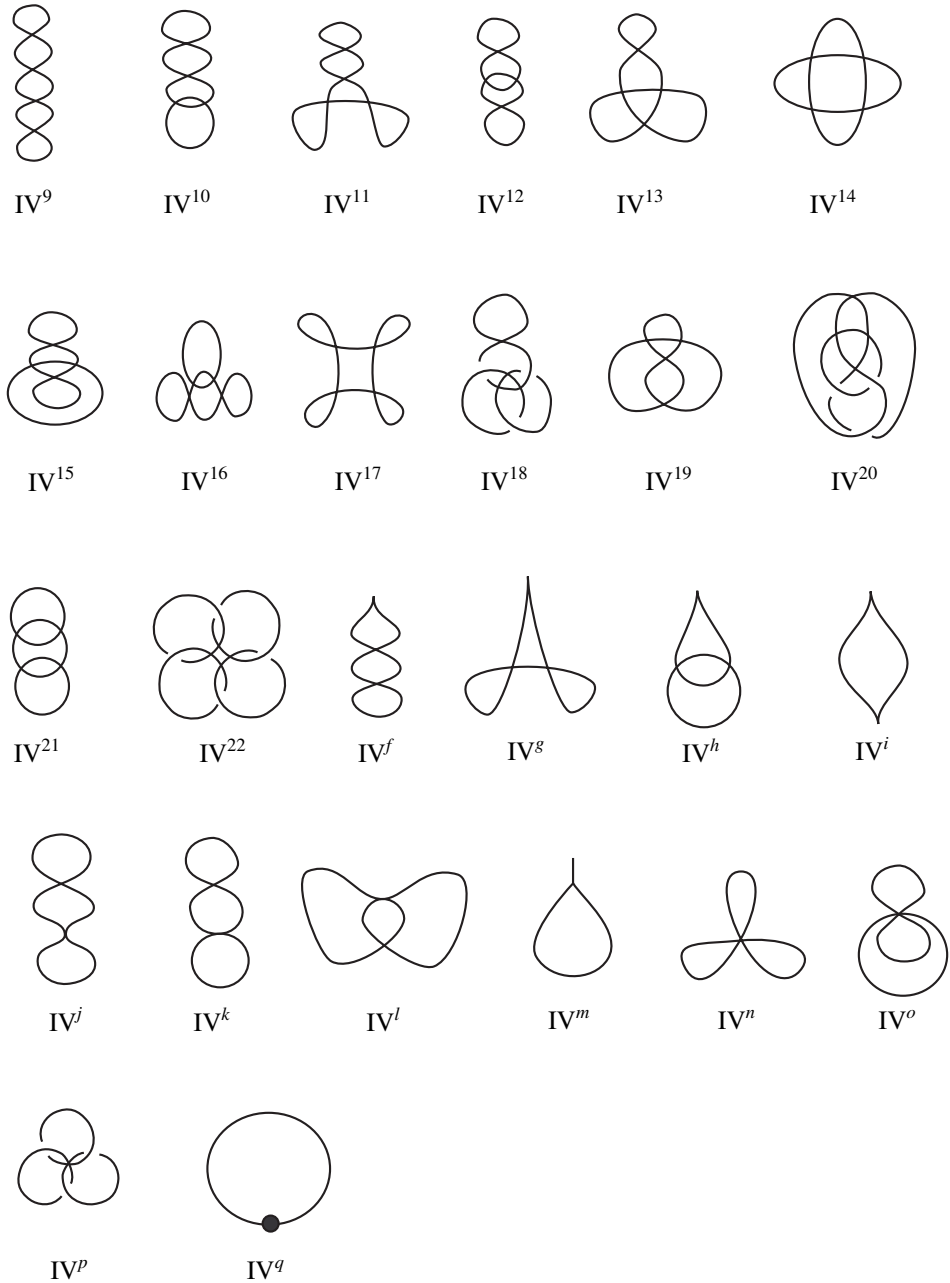

For the fibers in \fullref{thm:5classification}
(1), we use the same names as those of the
corresponding fibers in \fullref{thm:4classification}.
Note that the names of the fibers are 
consistent with the convention mentioned
in \fullref{rmk:name}.
Therefore, the figure corresponding to each fiber
listed in \fullref{thm:5classification} (2)
can be obtained by taking the disjoint union
of the fibers in \fullref{fig4} corresponding to
the digits or letters appearing in the superscript.
For example, the figure for the fiber $\mathrm{IV}^{0,0,0,1}$
consists of three dots and a ``figure $8$".

In \fullref{fig5}, we did not use ``$/$" as in
\fullref{fig4}, since the depicted fibers are
all connected and are easy to recognize.

Note also that the codimensions of
the fibers in \fullref{thm:5classification} (1)
coincide with those of the corresponding
fibers in \fullref{thm:4classification}.
Furthermore, the fibers in \fullref{thm:5classification} (2)
and (3) all have codimension $4$.

\begin{rem}
The result of \fullref{thm:5classification}
holds for the classification up to
$C^\infty$ equivalence as well. As a consequence, we see
that two fibers are $C^0$ equivalent if and only
if they are $C^\infty$ equivalent
(for related results, refer to \cite[Chapter~3]{Saeki04}).
This should be compared with a result of
Damon \cite{Damon} about stable map
germs in nice dimensions.
\end{rem}

\section{Chiral singular fibers and their signs}\label{section5}

In this section we determine those singular fibers
of proper stable maps of oriented $4$--manifolds
into $3$--manifolds which are chiral. 
We also define a sign $(= \pm 1)$ for
each chiral singular fiber of codimension $3$.

Let us first consider a fiber of type
$\mathrm{III}^8$.
Let $f \co (M, f^{-1}(y)) \to (N, y)$ be
a map germ representing the fiber of type $\mathrm{III}^8$
with $f^{-1}(y)$ being connected, where $M$ is an
orientable $4$--manifold and $N$ is a $3$--manifold. 
We assume that $M$ is
oriented. Let us denote the three
singular points of $f$ contained in $f^{-1}(y)$ by
$q_1, q_2$ and $q_3$. 

Let us fix an orientation of a neighborhood
of $y$ in $N$. Then for every regular
point $q \in f^{-1}(y)$, we can define the
local orientation of the fiber near $q$
by the ``fiber first" convention;
that is, we give the orientation
to the fiber at $q$ so that the ordered
$4$--tuple $\langle v, v_1, v_2, 
v_3\rangle$ of tangent vectors at $q$
gives the orientation of $M$, where
$v$ is a tangent vector of the fiber
at $q$ which corresponds to its orientation,
and $v_1$, $v_2$ and $v_3$ are tangent vectors
of $M$ at $q$ such that the ordered
$3$--tuple $\langle df_q(v_1), df_q(v_2), df_q(v_3)\rangle$ 
corresponds to the local orientation of $N$ at $y$.
Note that the set of regular points
in $f^{-1}(y)$ consists of six open arcs
and each of them gets its orientation by
the above rule. 

Each singular point $q_i$ is incident
to four open arcs. We see easily that
their orientations should be as depicted
in \fullref{local-ori} by considering
the orientations induced on the nearby fibers.

\begin{figure}[ht!]
  \centering
\labellist\small
\pinlabel {$q_i$} at 366 563
\endlabellist
\includegraphics[scale=0.3]{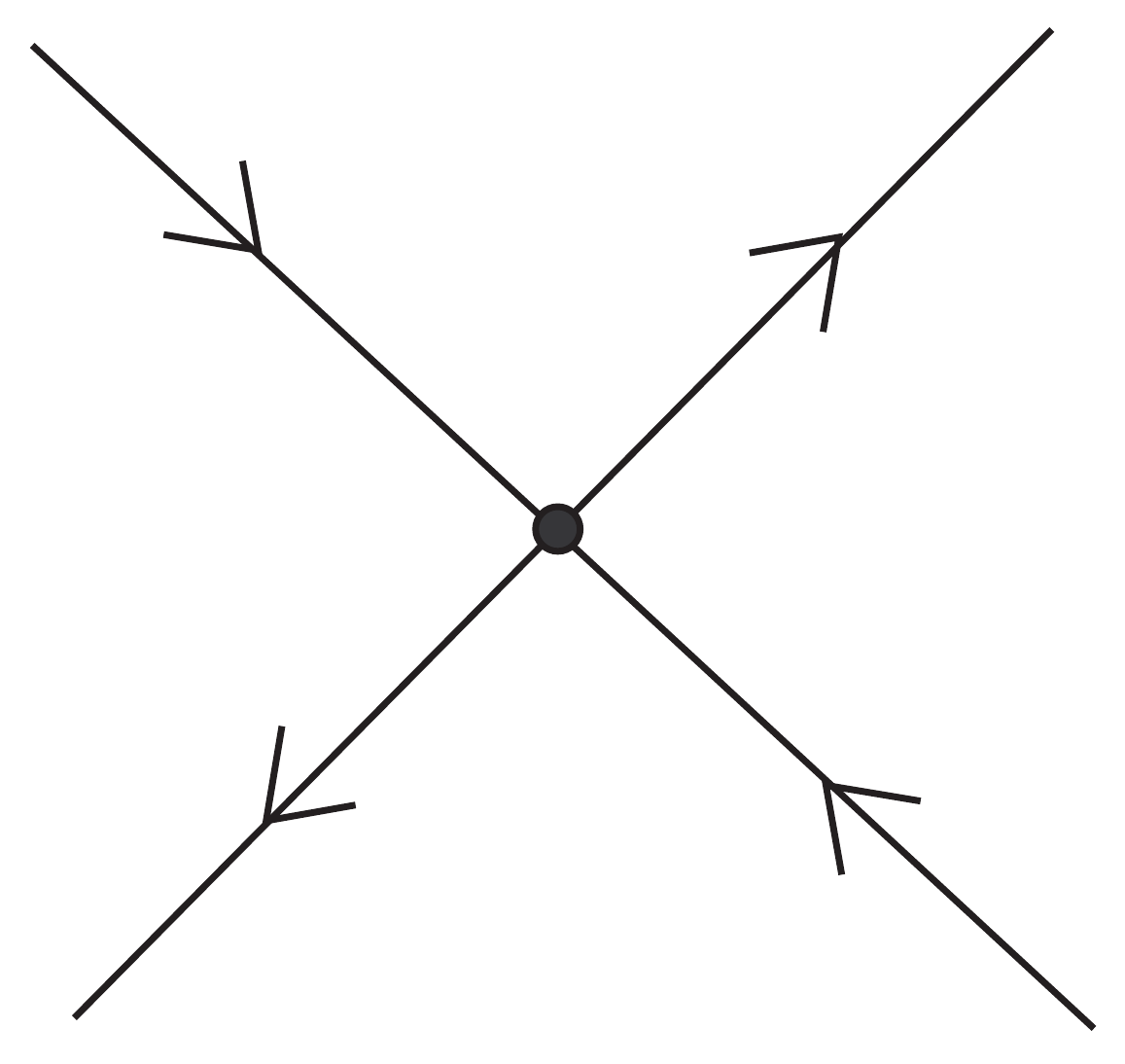}
\caption{Orientations of the four arcs incident to a singular point}
\label{local-ori}
\end{figure}

For each pair $(q_i, q_j)$, $i \neq j$, of
singular points, we
have exactly two open arcs of $f^{-1}(y)$ which connect
$q_i$ and $q_j$. Furthermore, the orientations
of the two open arcs coincide with each other
in the sense that
one of the two arcs goes from $q_i$ to $q_j$ if and
only if so does the other one.
Then we see that the orientations on the
six open arcs define a cyclic order of
the three singular points $q_1, q_2$ and $q_3$
(see \fullref{3ring1}).
By renaming the three singular points if necessary, 
we may assume that
this cyclic order is given by $\langle q_1,
q_2, q_3\rangle$.

\begin{figure}[ht!]
  \centering
\labellist\small
\hair=5pt
\pinlabel {$q_1$} [b] at 299 625
\pinlabel {$q_2$} [tr] at 260 544
\pinlabel {$q_3$} [tl] at 335 544
\endlabellist
\includegraphics[scale=0.3]{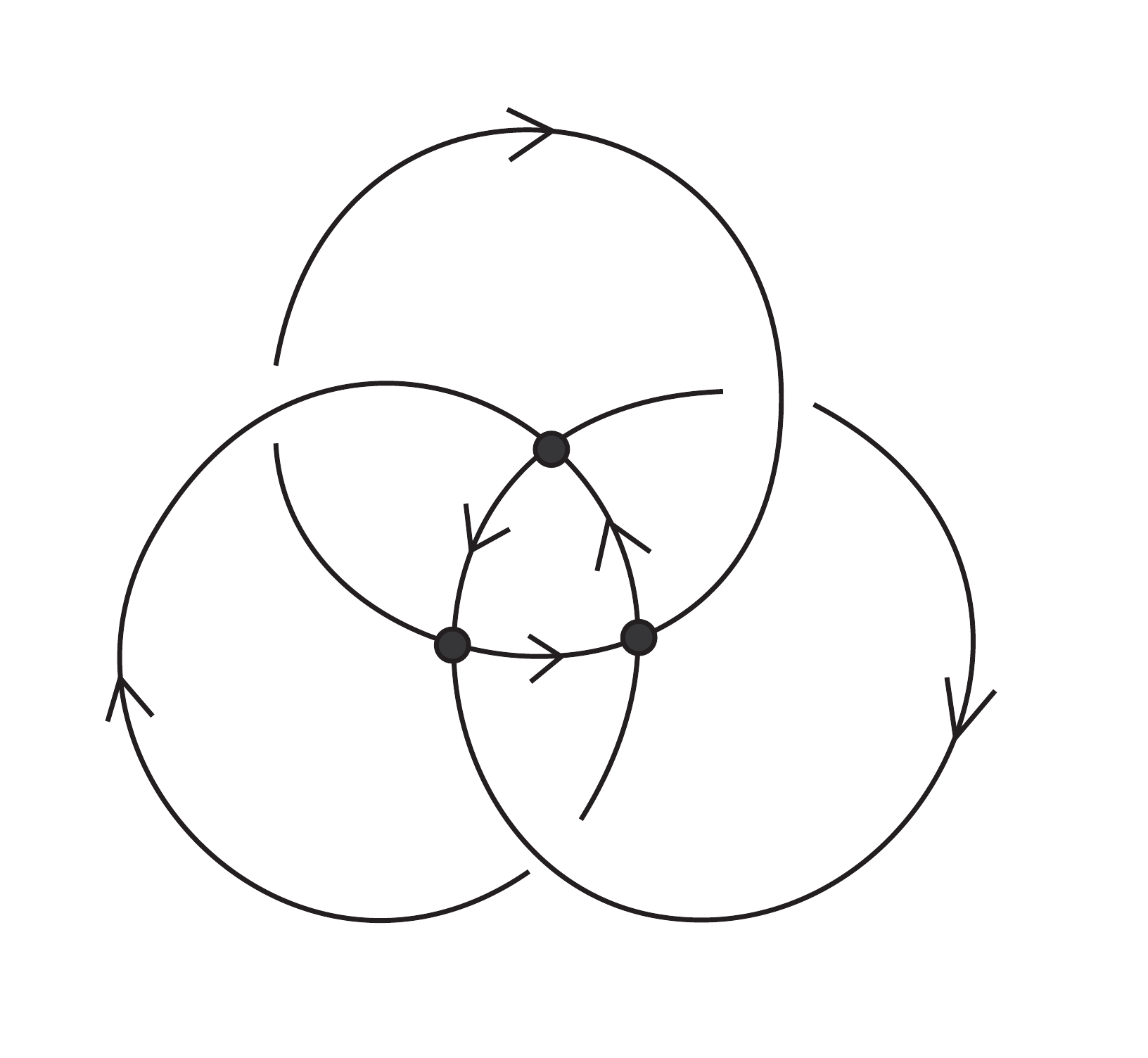}
\caption{Cyclic order of the three singular points}
\label{3ring1}
\end{figure}

Let $D_i$ be a sufficiently
small open disk neighborhood
of $q_i$ in $S(f)$. Since the multi-germ
$(f|_{S(f)}, f^{-1}(y) \cap S(f))$
corresponds to the triple point
as depicted in \fullref{fig3} (4),
the images $f(D_1), f(D_2)$ and $f(D_3)$
are open $2$--disks in $N$ in general
position forming a triple point at $y$.
They divide a neighborhood of $y$ in $N$
into eight octants. For each octant $\omega$,
take a point in it and count the number
of connected components of the
regular fiber over the point. It should
be equal either to $1$ or to $2$ and
it does not depend on the choice of the
point (for details, see \cite[Figure~3.6]{Saeki04}). 
When it is
equal to $k$ $(=1, 2)$, we call $\omega$
a \emph{$k$--octant}.

Choose a $1$--octant $\omega$.
Let $w_i$ be a normal
vector to $f(D_i)$ in $N$ pointing toward
$\omega$ at a point
incident to that octant, $i = 1, 2, 3$
(see \fullref{triple0}).

\begin{figure}[ht!]
\centering
\labellist\small
\pinlabel {$f(D_1)$} [bl] at 434 708
\pinlabel {$f(D_2)$} [br] at 109 696
\pinlabel {$f(D_3)$} [r] at 121 400
\pinlabel {$w_1$} [b] at 323 625
\pinlabel {$w_2$} [t] at 217 584
\pinlabel {$w_3$} [l] at 298 521
\pinlabel {$y$} [l] at 457 638
\pinlabel {$\omega$} [br] at 46 520
\endlabellist
\includegraphics[scale=0.4]{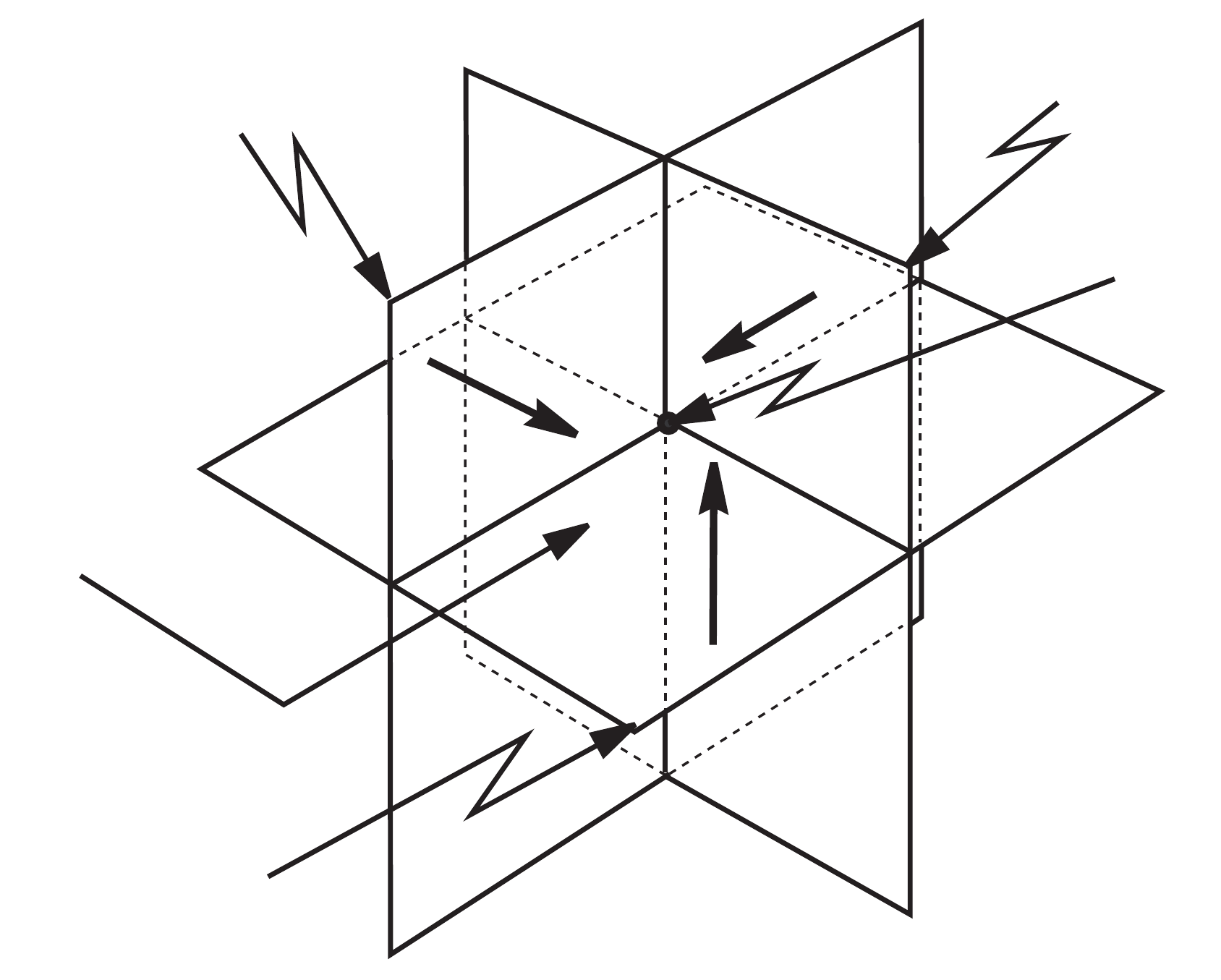}
\caption{Vectors $w_i$ normal to $f(D_i)$
pointing toward $\omega$}
\label{triple0}
\end{figure}
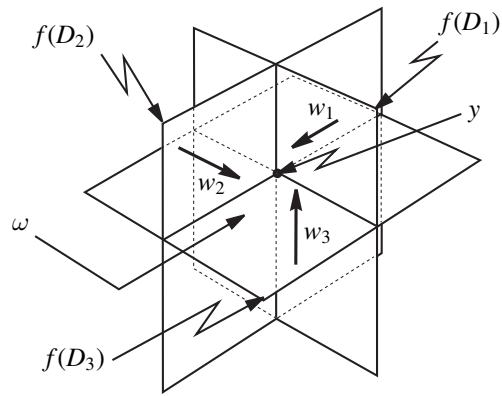

We may identify a neighborhood of $y$ in $N$ with $\R^3$.
Then the local orientation at $y$
corresponding to the ordered
$3$--tuple of vectors $\langle w_1, w_2, w_3\rangle$
depends only on the cyclic order of the
three open disks $f(D_1)$, $f(D_2)$ and $f(D_3)$
and is well-defined, once a $1$--octant is chosen.
Then we say that the fiber $f^{-1}(y)$
is \emph{positive} if the orientation
corresponding to $\langle w_1, w_2, w_3\rangle$
coincides with the local orientation of $N$
at $y$ which we chose at the beginning; otherwise, \emph{negative}.
We define the \emph{sign} of the fiber
to be $+1$ (or $-1$) if it is positive
(respectively, negative).

\begin{lem}
The above definition does not depend
on the choices of the following data, and
the sign of a $\mathrm{III}^8$ type fiber
is well-defined as long as the source
$4$--manifold is oriented:
\begin{itemize}
\item[$(1)$] the $1$--octant $\omega$,
\item[$(2)$] the local orientation of $N$ at $y$.
\end{itemize}
\end{lem}

\begin{proof}
$(1)$ It is easy to see that
any two adjacent octants have distinct
numbers of connected components of their
associated regular fibers; that is, a $1$--octant
is adjacent to $2$--octants, but never to
another $1$--octant, and vice versa. Therefore, in order to
move from the chosen $1$--octant to another
$1$--octant,
one has to cross the open disks $f(D_1),
f(D_2)$ and $f(D_3)$ even number of times.
Every time one crosses an open disk,
the associated normal vector corresponding to
that open disk changes the direction, while
the other two vectors remain parallel.
Therefore, after crossing the open disks
even number of times, we get the
same orientation determined by the associated
ordered normal vectors.

$(2)$ If we reverse the local orientation
of $N$ at $y$, then the regular parts
of fibers get opposite orientations.
Therefore, in the above definition,
the cyclic order of the three singular
points is reversed. Hence, the resulting
local orientation at $y$ determined by the
three normal vectors is also reversed.
Thus, the sign of the fiber is well-defined.
\end{proof}

For a fiber which is a disjoint union
of a $\mathrm{III}^8$ type fiber and a finite
number of copies of a fiber of the trivial
circle bundle (that is, for a fiber
equivalent to a $\mathrm{III}^8$ type fiber
modulo regular fibers), we say that it is positive
(respectively, negative) if the $\mathrm{III}^8$--fiber
component is positive (respectively, negative).
We define the \emph{sign} of such a fiber
to be $+1$ (or $-1$) if it is positive
(respectively, negative).

\begin{rem}\label{rem:ori-rev}
It should be noted that if we reverse the
orientation of the source $4$--manifold, then
the sign of a $\mathrm{III}^8$ type fiber necessarily
changes.
\end{rem}

\begin{cor}
A fiber equivalent to a $\mathrm{III}^8$ type fiber
modulo regular fibers is always chiral.
\end{cor}

\begin{proof}
If it is achiral, then a representative
of a $\mathrm{III}^8$ type singular fiber
and its copy with the orientation of the
source $4$--manifold being reversed are
$C^0$ equivalent with respect to an
orientation preserving homeomorphism
between the sources
(that is, with respect to a homeomorphism $\tilde{\varphi}$
as in the diagram \eqref{eq:chiral}).
Let us take
local orientations at the target points
so that the homeomorphism between the target
manifolds (that is, the homeomorphism
$\varphi$ in the diagram \eqref{eq:chiral})
preserves the orientation.
Then by our definition of the sign,
we see that the two $\mathrm{III}^8$ type fibers
should have the same sign, which is a contradiction
in view of \fullref{rem:ori-rev}.
Therefore, the desired conclusion follows.
This completes the proof.
\end{proof}

Let us now consider the other singular fibers
appearing in \fullref{thm:4classification}.
By using similar arguments, we can determine the
chiral singular fibers among the list.
More precisely, we have the following.

\begin{prop}\label{prop:4chiral}
A singular fiber of a proper $C^\infty$
stable map of an oriented $4$--manifold
into a $3$--manifold is chiral
if and only if it contains a fiber
of type
$\mathrm{III}^5$, $\mathrm{III}^7$ or
$\mathrm{III}^8$.
\end{prop}

\begin{proof}
For fibers of types $\mathrm{III}^5$ and $\mathrm{III}^7$,
we can define their signs as for 
a $\mathrm{III}^8$ type fiber. 
Therefore, they are chiral.
Details are left to the reader.

For the other fibers,
we can find homeomorphisms 
$\tilde{\varphi}$ and $\varphi$ as in 
\fullref{dfn:achiral}.
For example, let us consider a 
$\mathrm{II}^2$ type fiber. Let 
$f \co (M, f^{-1}(y)) \to (N, y)$
be a proper smooth map germ
representing a fiber of
type $\mathrm{II}^2$, and
let $q_1$ and $q_2$ be the two singular points
contained in $f^{-1}(y)$, both of
which are indefinite fold points.
We fix orientations of $M$ and $N$
near $f^{-1}(y)$ and $y$ respectively.
Then the regular part of $f^{-1}(y)$
is naturally oriented by the ``fiber
first" convention. 

It is easy to show that the
involution of $f^{-1}(y)$ as in \fullref{inv1}
reverses the orientation of the regular
part of $f^{-1}(y)$. Note that this
involution fixes the two singular points
$q_1$ and $q_2$ pointwise.

\begin{figure}[ht!]
  \centering
  \includegraphics[scale=0.3]{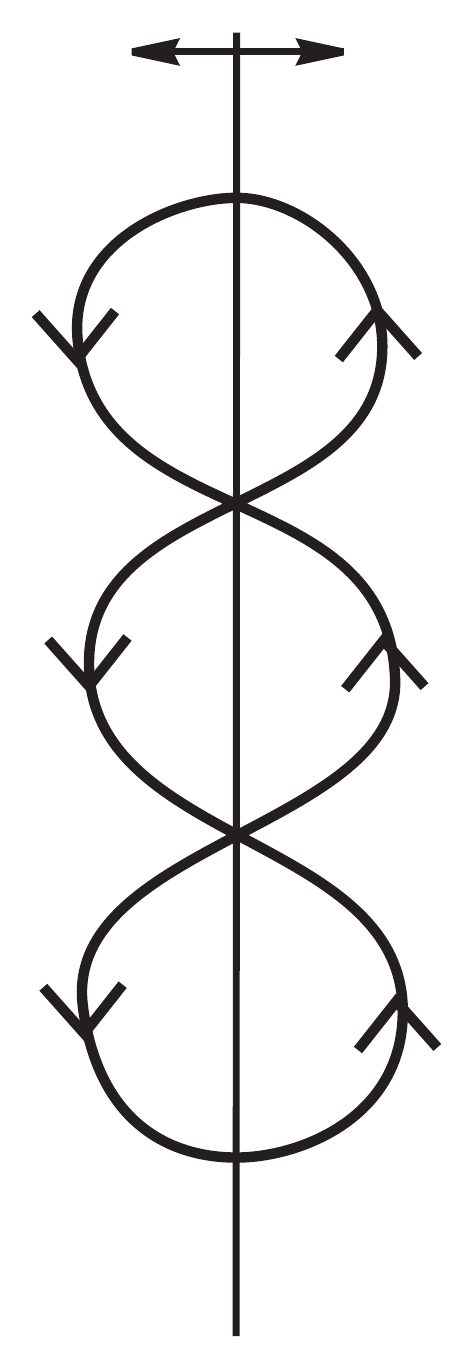}
  \caption{Orientation-reversing involution of a $\mathrm{II}^2$ 
  type fiber}
  \label{inv1}
\end{figure}

By \fullref{section3}, there exist
coordinates $(x_i, y_i, z_i, w_i)$ and
$(X, Y, Z)$ around $q_i$, $i = 1, 2$, and $f(q_i)
= y$ respectively such that $f$ is given by
$$(x_1, y_1, z_1, w_1) \mapsto (x_1, y_1, z_1^2 - w_1^2)$$
around $q_1$, and by
$$(x_2, y_2, z_2, w_2) \mapsto (x_2, z_2^2 - w_2^2, y_2)$$
around $q_2$ with respect to these coordinates.
Then we may assume that the above involution
is consistent with the involutions
defined by
$$(x_1, y_1, z_1, w_1) \mapsto (x_1, y_1, -z_1, w_1)$$
around $q_1$ and by
$$(x_2, y_2, z_2, w_2) \mapsto (x_2, y_2, -z_2, z_2)$$
around $q_2$. Then we can extend this involution
of a neighborhood of $\{q_1, q_2\}$ to a self-diffeomorphism
$\tilde{\varphi}$ of $f^{-1}(U)$ for a sufficiently small open disk
neighborhood $U$ of $y$ in $N$ so that the diagram
$$\bfig
\barrsquare<1500,400>[(f^{-1}(U), f^{-1}(y))`
    (f^{-1}(U), f^{-1}(y))`
    (U, y)`
    (U, y);\tilde{\varphi}`f`f`\id_U]
\efig$$
is commutative, by using the relative version
of Ehresmann's fibration theorem
(see Ehresmann \cite{E}, Lamotke \cite[Section~3]{L}, Br\"ocker--J\"anich
\cite[Section~8.12]{BJ}, or the book \cite[Section~1]{Saeki04} by the
first named author), where $\id_U$ denotes the
identity map of $U$.
Note that the diffeomorphism $\tilde{\varphi}$
thus constructed is orientation reversing.
Hence, the fiber $f^{-1}(y)$ is
achiral according to \fullref{dfn:achiral}.

We can use similar arguments for the other fibers
to show that they are achiral. Details are
left to the reader.
\end{proof}

Let us now state the main theorem of this paper.
For a closed oriented $4$--manifold, we denote
by $\sigma(M)$ the signature of $M$. Furthermore,
for a $C^\infty$ stable map $f \co M \to N$
into a $3$--manifold $N$,
we denote by $||\mathrm{III}^8(f)||$ the
algebraic number of $\mathrm{III}^8$ type fibers of $f$;
that is, it is the sum of the signs over all
fibers of $f$ equivalent to a
$\mathrm{III}^8$ type fiber modulo regular
fibers.

\begin{thm}\label{thm:main}
Let $M$ be a closed oriented $4$--manifold and
$N$ a $3$--manifold. Then, for any $C^\infty$ stable map
$f \co M \to N$, we have 
$$\sigma(M) = ||\mathrm{III}^8(f)|| \in \Z.$$
\end{thm}

The proof of \fullref{thm:main}
will be given in \fullref{section7}.

Since for an oriented
$4$--manifold, the signature and the
Euler characteristic have the same
parity, we immediately obtain the following,
which was obtained in \cite{Saeki04}.

\begin{cor}\label{cor:Euler}
Let $M$ be a closed orientable $4$--manifold and
$N$ a $3$--manifold. Then, for any $C^\infty$ stable map
$f \co M \to N$,
the number of fibers of $f$ equivalent to
a $\mathrm{III}^8$ type fiber modulo regular
fibers has the same parity
as the Euler characteristic of $M$.
\end{cor}

Note that in the proof of our main theorem,
we do not use the above corollary. In other
words, our proof gives a new proof for
the above modulo two Euler characteristic
formula.

\section{Cobordism invariance of the algebraic number of
$\mathrm{III}^8$ type fibers}\label{section6}

In order to prove \fullref{thm:main},
let us first show that the algebraic number
of $\mathrm{III}^8$ type fibers is an oriented
cobordism invariant of the source $4$--manifold.

Let us begin by a list of chiral singular fibers
of proper $C^\infty$ stable maps of $5$--manifolds
into $4$--manifolds. We can prove the following
proposition by an argument similar to that
in the previous section.

\begin{prop}\label{prop:5chiral}
A singular fiber of a proper $C^\infty$
stable map of an oriented $5$--manifold
into a $4$--manifold is chiral
if and only if it is 
$C^0$ equivalent modulo regular fibers to
a fiber of type
$\mathrm{III}^5$, $\mathrm{III}^7$,
$\mathrm{III}^8$,
$\mathrm{IV}^{0,5}$, $\mathrm{IV}^{0,7}$, 
$\mathrm{IV}^{0,8}$, $\mathrm{IV}^{1,5}$, 
$\mathrm{IV}^{1,7}$, $\mathrm{IV}^{1,8}$, 
$\mathrm{IV}^{10}$, $\mathrm{IV}^{11}$, 
$\mathrm{IV}^{12}$, $\mathrm{IV}^{13}$, 
$\mathrm{IV}^{18}$, $\mathrm{IV}^{g}$, 
$\mathrm{IV}^{h}$, or $\mathrm{IV}^{k}$.
\end{prop}

For example, in order to show that the
fibers of types $\mathrm{IV}^o$,
$\mathrm{IV}^p$ and $\mathrm{IV}^q$
are achiral, we can use the symmetry
of order $6$ of an indefinite $D_4$ point
as in \fullref{rem:symmetry}.
The proof of \fullref{prop:5chiral}
is left to the reader.

Note that for each chiral singular fiber
of codimension $4$, we can define its
sign $(= \pm 1)$, as long as the source $5$--manifold is
oriented. In what follows, we fix such a definition
of a sign for each chiral singular fiber of codimension $4$
once and for all, although we do not mention it
explicitly.

Let $\mathfrak{F}$ be a $C^0$ equivalence
class modulo regular fibers.
For a proper $C^\infty$ stable map
$f \co M \to N$ of an oriented $5$--manifold $M$
into a $4$--manifold $N$, we denote by
$\mathfrak{F}(f)$ the set of all $y \in N$
over which lies a fiber of type $\mathfrak{F}$.
Note that $\mathfrak{F}(f)$ is a regular $C^\infty$
submanifold of $N$ of codimension $\kappa(\mathfrak{F})$,
where $\kappa(\mathfrak{F})$ denotes the
codimension of the $C^0$ equivalence
class modulo regular fibers $\mathfrak{F}$.

In general, if $\mathfrak{F}$ is chiral, then
$\mathfrak{F}(f)$ is orientable. For 
$\mathfrak{F} = \mathrm{III}^8$, we introduce
the orientation on $\mathrm{III}^8(f)$ as follows.

Take a point $y \in \mathrm{III}^8(f)$. 
Note that the singular value set $f(S(f))$ near
$y$ consists of three codimension $1$ ``sheets"
meeting along $\mathrm{III}^8(f)$ in general
position. Let $D_y$
be a small open $3$--disk centered at $y$ in $N$
which intersects $\mathrm{III}^8(f)$ transversely
exactly at $y$ and is transverse to the
three sheets of $f(S(f))$. Put $M' = f^{-1}(D_y)$,
which is a smooth $4$--dimensional submanifold
of $M$ with trivial normal bundle and
hence is orientable.
Let us consider the proper smooth map
$$h = f|_{f^{-1}(D_y)} \co M' \to D_y,$$
which is a $C^\infty$ stable map by
virtue of \fullref{prop:stable}. 
Note that the fiber of $h$
over $y$ is of type $\mathrm{III}^8$.
Let $M''$ be the component of $M'$ containing
the $\mathrm{III}^8$ type fiber.

Let $\mathcal{O}_{M''}$ be the orientation
of $M''$ with respect to which the
$\mathrm{III}^8$ type fiber is positive.
Then let $\mathcal{O}_{\nu}$ be the
orientation of the normal bundle $\nu$ to 
$M''$ in $M$ such that $\mathcal{O}_{\nu}
\oplus \mathcal{O}_{M''}$ is consistent
with the orientation of the $5$--manifold $M$.
By the differential $df \co TM \to TN$ at a point in $M''$,
$\mathcal{O}_{\nu}$ corresponds to
a normal direction to $D_y$ in $N$ at $y$.
Now we orient $\mathrm{III}^8(f)$ at $y$ so that
this direction is consistent with the
orientation of $\mathrm{III}^8(f)$.

It is easy to see that this orientation
varies continuously with respect to $y
\in \mathrm{III}^8(f)$ and hence
defines an orientation on $\mathrm{III}^8(f)$.

Now let $\mathfrak{F}$ be 
the $C^0$ equivalence
class modulo regular fibers
of one of the codimension $4$ fibers 
appearing in \fullref{prop:5chiral};
that is, $\mathfrak{F}$ is a chiral singular fiber
of codimension $4$.
Note that $\mathfrak{F}(f)$ is a discrete set
in $N$. Take a point $y \in \mathfrak{F}(f)$
and a sufficiently small open
disk neighborhood $\Delta_y$ of $y$
in $N$. We orient the source $5$--manifold
so that the fiber over $y$ gets the sign $+1$.
Then $\Delta_y \cap \mathrm{III}^8(f)$
consists of several oriented arcs which have
a common end point at $y$. Let us
define the \emph{incidence coefficient}
$[\mathrm{III}^8:\mathfrak{F}] \in \Z$ to be
the number of arcs coming into $y$
minus the number of arcs going out
of $y$. Note that this does not
depend on the point $y$ nor on the map $f$.

\begin{rem}\label{rem:incidence-achiral}
Let $\mathfrak{F}$ be the $C^0$
equivalence class modulo regular fibers
of a codimension $4$ \emph{achiral} singular fiber.
Then we can define the incidence coefficient
$[\mathrm{III}^8:\mathfrak{F}] \in \Z$
in exactly the same manner as above.
However, this should always vanish, since
the homeomorphism $\varphi$ as in \eqref{eq:chiral}
reverses the orientation of
$\Delta_y \cap \mathrm{III}^8(f)$.
\end{rem}

\begin{lem}\label{lem:incidence}
The incidence coefficient $[\mathrm{III}^8:\mathfrak{F}]$
vanishes for every $C^0$ equivalence class
modulo regular fibers $\mathfrak{F}$
of codimension $4$ that is chiral.
\end{lem}

\begin{proof}
It is not difficult to see that
for $y \in \mathfrak{F}(f)$,
$\Delta_y \cap \mathrm{III}^8(f) \neq \emptyset$
if and only if $\mathfrak{F} = \mathrm{IV}^{0, 8}$,
$\mathrm{IV}^{1, 8}$ or $\mathrm{IV}^{18}$.
Furthermore, for each of these three cases,
the number of arcs of $\Delta_y \cap \mathrm{III}^8(f)$
is equal to $2$ and exactly
one of them is coming into $y$. Thus the
result follows. 
\end{proof}

\begin{rem}\label{rem:submanifold}
The above lemma shows that the closure
of $\mathrm{III}^8(f)$ is a
regular oriented $1$--dimensional submanifold of $N$
near the points over which lies a chiral singular fiber
of codimension $4$. However, the closure of
$\mathrm{III}^8(f)$, as a whole,
is not even a topological manifold in general. For example,
suppose that $f$ admits a $\mathrm{IV}^{22}$ type fiber.
Then the closure of $\mathrm{III}^8(f)$
forms a graph (that is, a $1$--dimensional
complex) and each point
of $\mathrm{IV}^{22}(f)$ is a vertex
of degree $8$, that is, it has
$8$ incident edges. Furthermore, four of them
are incoming edges and the other
four are outgoing edges.
\end{rem}

Let us recall the following definition
(for details, refer to Conner--Floyd \cite{CF}.)

\begin{dfn}
Let $N$ be a manifold and
$f_i \co M_i \to N$ a
continuous map of a closed oriented $n$--dimensional
manifold $M_i$ into $N$, $i = 0, 1$.
We say that $f_0$ and $f_1$ are \emph{oriented
bordant} if there exist a compact oriented 
$(n+1)$--dimensional manifold $W$ 
and a continuous map $F \co W \to N \times [0, 1]$ with
the following properties:
\begin{enumerate}
\item $\partial W$ is identified with the disjoint union
of $-M_0$ and $M_1$, where $-M_0$ denotes
the manifold $M_0$ with the reversed orientation, and
\item $F|_{M_i} \co M_i \to N \times \{i\}$ 
is identified with $f_i$, $i = 0, 1$.
\end{enumerate}
We call the map $F \co W \to N \times [0, 1]$
an \emph{oriented bordism} between $f_0$ and $f_1$.

Note that if $M_0 = M_1$, and $f_0$ and $f_1$
are homotopic, then they are oriented bordant.
Furthermore, if the target manifold $N$ is
contractible, then $f_0$ and $f_1$ are
oriented bordant if and only if their
source manifolds $M_0$
and $M_1$ are oriented cobordant
as oriented manifolds.

For a given manifold $N$ and a nonnegative integer $n$,
the set of all oriented bordism classes of maps
of closed oriented $n$--dimensional manifolds into $N$
forms an additive group under the disjoint union.
We call it the \emph{$n$--dimensional oriented
bordism group of $N$}.
\end{dfn}

Note that in the usual definition,
an oriented bordism is a map
into $N$ and not into $N
\times [0, 1]$. However, it is easy to see that
the above definition is equivalent
to the usual one.

As a consequence of \fullref{lem:incidence}, 
we get the following.

\begin{lem}\label{lem:bord-inv}
Let $N$ be a $3$--manifold and
$f_i \co M_i \to N$ a $C^\infty$
stable map of a closed oriented $4$--manifold
$M_i$ into $N$, $i = 0, 1$. If $f_0$
and $f_1$ are oriented bordant, then
we have
$$||\mathrm{III}^8(f_0)|| = ||\mathrm{III}^8(f_1)||.$$
\end{lem}

\begin{proof}
Let $F \co W \to N \times [0, 1]$ be an
oriented bordism between $f_0$ and $f_1$.
Take sufficiently small
collar neighborhoods $C_0 = M_0 \times
[0, 1)$ and $C_1 = M_1 \times (0, 1]$
of $M_0$ and $M_1$ in $W$ respectively.
We may assume that 
\begin{eqnarray*}
F|_{M_0 \times [0, \varepsilon)} 
& = & f_0 \times \id_{[0, \varepsilon)}, \quad \text{and} \\
F|_{M_1 \times (1-\varepsilon, 1]} & 
= & f_1 \times \id_{(1-\varepsilon, 1]}
\end{eqnarray*}
for a sufficiently small $\varepsilon > 0$.
Furthermore, we may assume that $F$ is a
smooth map with $F^{-1}(N \times (0,1)) = \Int{W}$.
Then by a standard argument, we can
approximate $F$ by a generic map $F'$
such that $F'|_{C_0 \cup C_1} = F|_{C_0 \cup C_1}$
and that $F'|_{\Int{W}} \co \Int{W} \to N \times (0, 1)$
is a proper $C^\infty$
stable map. In the following,
let us denote $F'$ again by $F$.

By \fullref{lem:incidence}, we see
that the closure of
$\mathrm{III}^8(F)$ is a 
finite graph each of whose edge is
oriented. Furthermore, for each
vertex lying in $N \times (0, 1)$,
the number of incoming edges is equal
to that of outgoing edges.
Furthermore, its vertices lying
in $N \times \{0, 1\}$
have degree one
and they coincide exactly with the union of
\begin{eqnarray*}
\mathrm{III}^8(F) \cap (N \times \{0\}) 
& = & \mathrm{III}^8(f_0) 
\quad \text{and} \\
\mathrm{III}^8(F) \cap (N \times \{1\}) 
& = & \mathrm{III}^8(f_1).
\end{eqnarray*}
Therefore, by virtue of \fullref{rem:ori-rev}
we have
$$-||\mathrm{III}^8(f_0)|| + ||\mathrm{III}^8(f_1)||
= 0,$$
since $\partial W = (-M_0) \cup M_1$.
Hence the result follows.
\end{proof}

By combining \fullref{lem:bord-inv} with
a work of Conner--Floyd \cite{CF}, we
get the following.

\begin{prop}\label{prop:cobord-inv}
Let $N$ be a $3$--manifold
and $f_i \co M_i \to N$ a $C^\infty$ stable map
of a closed oriented $4$--manifold $M_i$ into $N$,
$i = 0, 1$. If $M_0$ and $M_1$ are
oriented cobordant as oriented $4$--manifolds,
then we have
$$||\mathrm{III}^8(f_0)|| = ||\mathrm{III}^8(f_1)||.$$
\end{prop}

\begin{proof}
Recall that the oriented cobordism groups
$\Omega_n$ of $n$--dimensional manifolds
for $0 \leq n \leq 4$ satisfy the following:
$$
\Omega_n \cong
\begin{cases}
0, & n = 1, 2, 3, \\
\Z, & n = 0, 4.
\end{cases}
$$
Furthermore, the $4$--dimensional oriented
bordism group of $N$ is isomorphic to
$$\sum_{p+q=4} H_p(N; \Omega_q)$$
modulo (odd) torsion \cite[Section~15]{CF}. 
Therefore, if the $4$--dimensional manifolds
$M_0$ and $M_1$ are oriented cobordant,
then $m f_0$ and $m f_1$ are oriented
bordant for some odd integer $m$, where
$m f_i$ denotes the map
of the disjoint union of $m$ copies of $M_i$
into $N$ such that on each copy it is given by
$f_i$, $i = 0, 1$.

Thus by \fullref{lem:bord-inv}, we have 
$$m||\mathrm{III}^8(f_0)|| = m||\mathrm{III}^8(f_1)||,$$
which implies the desired equality.
This completes the proof.
\end{proof}

\section{An explicit example}\label{section7}

In this section, we study an explicit
example of a $C^\infty$ stable map
of a closed oriented $4$--manifold
with nonzero signature into $\R^3$.
Combining this with \fullref{prop:cobord-inv},
we will prove our main theorem of this paper.

In \cite[Chapter~6]{Saeki04}, the first named
author constructed an explicit example
of a $C^\infty$ stable map $f \co M
\to \R^3$ of a closed $4$--manifold $M$
with exactly one $\mathrm{III}^8$ type fiber
such that $f$ has only fold points as its
singularities. Recall that $M$ is
diffeomorphic to $\C P^2 \sharp 2 \overline{\C P^2}$
if we ignore the orientation.
We would like to orient $M$ and
determine the sign of this $\mathrm{III}^8$ type fiber.

In fact, by \fullref{prop:cobord-inv},
we already know that there exists
a constant $c$ such that the algebraic number
of $\mathrm{III}^8$ type fibers is $c$
times the signature of the source oriented
$4$--manifold (for details, see the
argument in the proof of \fullref{thm:main}
below). In the above-mentioned example, the algebraic
number of $\mathrm{III}^8$ type fibers is
equal to $\pm 1$, and the signature of the
source $4$--manifold is equal to $\pm 1$.
Therefore, this constant $c$ should be
equal to $\pm 1$. Thus, for the proof of
our main theorem, it suffices to
determine the sign of the constant $c$.

This procedure might seem easy, but in
fact it is not. As a matter of fact,
the construction of the above example
in \cite{Saeki04} was already very
complicated, although the example
itself seems to be a natural one. 
Therefore, in this section, we carefully
study the example and determine the sign
of the constant $c$. We will describe
the argument in details, since the
technique in this section can be
very useful in determining the self-intersection
number of the surface of singular point set
in general situations. At the end of
this section, we give a new proof of the
self-intersection number formula based on
our study.

In the construction given in \cite{Saeki04},
the orientation of the source $4$--manifold $M$
was not given explicitly. 
Here, we first orient
the source $4$--manifold so that the $\mathrm{III}^8$ type
fiber gets the positive sign, and then determine
the signature of the source $4$--manifold
with respect to the chosen orientation.

\begin{lem}\label{lem:example}
If we orient the source $4$--manifold
$M$ of the above example so that
the $\mathrm{III}^8$ type fiber is positive,
then the signature of $M$
is equal to $+1$.
\end{lem}

\begin{proof}
In general, let $f \co M \to \R^3$ be
a $C^\infty$ stable map of a closed
oriented $4$--manifold into the $3$--dimensional
Euclidean space
which has only fold points as its singularities.
In view of a result in the first named author's paper \cite{Saeki02} (see
also Ohmoto--Saeki--Sakuma \cite{OSS} or Sadykov \cite{Sadykov}), we have
\begin{equation}
3\sigma(M) = S_0(f) \cdot S_0(f),
\label{eq:self-intersection}
\end{equation}
where $S_0(f)$ denotes the surface of definite
fold singularities of $f$, and $S_0(f) \cdot S_0(f)$
denotes its self-intersection number (or
equivalently, its normal Euler number) in $M$.
Therefore, in order to determine the
signature of the source $4$--manifold
of the explicit example mentioned above,
we have only to compute the self-intersection
number $S_0(f) \cdot S_0(f)$.

Let us first consider the $\mathrm{III}^8$
type fiber and denote the 
three singular points in
it by $q_1$, $q_2$ and $q_3$
as in \fullref{3ring1}.
Furthermore, we
orient the regular part of the $\mathrm{III}^8$
type fiber
so that it corresponds to the cyclic order
$\langle q_1, q_2, q_3\rangle$ of the three singular
points.

The image $f(S(f))$ of the singular point
set of $f$ around the point $y$ corresponding
to the $\mathrm{III}^8$ type fiber consists
of three sheets $f(D_1)$, $f(D_2)$ and $f(D_3)$,
where $D_i$ is a small disk neighborhood
of $q_i$ in $S(f)$, $i = 1, 2, 3$. 
We may assume that the three
sheets $f(D_1)$, $f(D_2)$ and $f(D_3)$ 
are situated as in
\fullref{triple0} and that the ``front
octant" $\omega$ in the figure is a $1$--octant; 
that is, the fiber
over a point in the octant is connected.

Let $w_i$ be a normal
vector to the $i$-th sheet of $f(S(f))$ 
in $N$ pointing toward
$\omega$ at a point
incident to that octant, $i = 1, 2, 3$.
We orient $\R^3$ so that the ordered $3$--tuple of
vectors $\langle w_1, w_2, w_3 \rangle$
is consistent with the orientation, that is,
$\R^3$ is endowed with the ``right-handed
orientation" in the usual sense.

Now we orient the source $4$--manifold of $f$
so that the regular part
of the $\mathrm{III}^8$ type fiber gets the
orientation as indicated in \fullref{3ring1}
in the sense of
\fullref{section5} by the ``fiber first"
convention. Then the sign of the 
$\mathrm{III}^8$ type fiber is equal to $+1$.

Recall that $S_0(f)$ consists
of three $2$--spheres $\tilde{S}_0$, 
$\tilde{S}_1$ and $\tilde{S}_2$, and that the
surface $S_1(f)$ of indefinite fold points
is a real projective plane whose image
$P = f(S_1(f))$ is Boy's surface in $\R^3$
(see \fullref{Boy}).
Furthermore, the embedded $2$--sphere
$S_0 = f(\tilde{S}_0)$ surrounds Boy's surface, and
the disjoint $2$--spheres
$S_1 = f(\tilde{S}_1)$ and 
$S_2 = f(\tilde{S}_2)$ are contained in the
bounded region of $\R^3 \setminus P$
so that $S_1$ surrounds $S_2$ (for details, see
\cite[Chapter~6]{Saeki04}).

\begin{figure}[ht!]
\centering
\labellist\small
\pinlabel {attach} [b] at 330 636
\endlabellist
\includegraphics[scale=0.3]{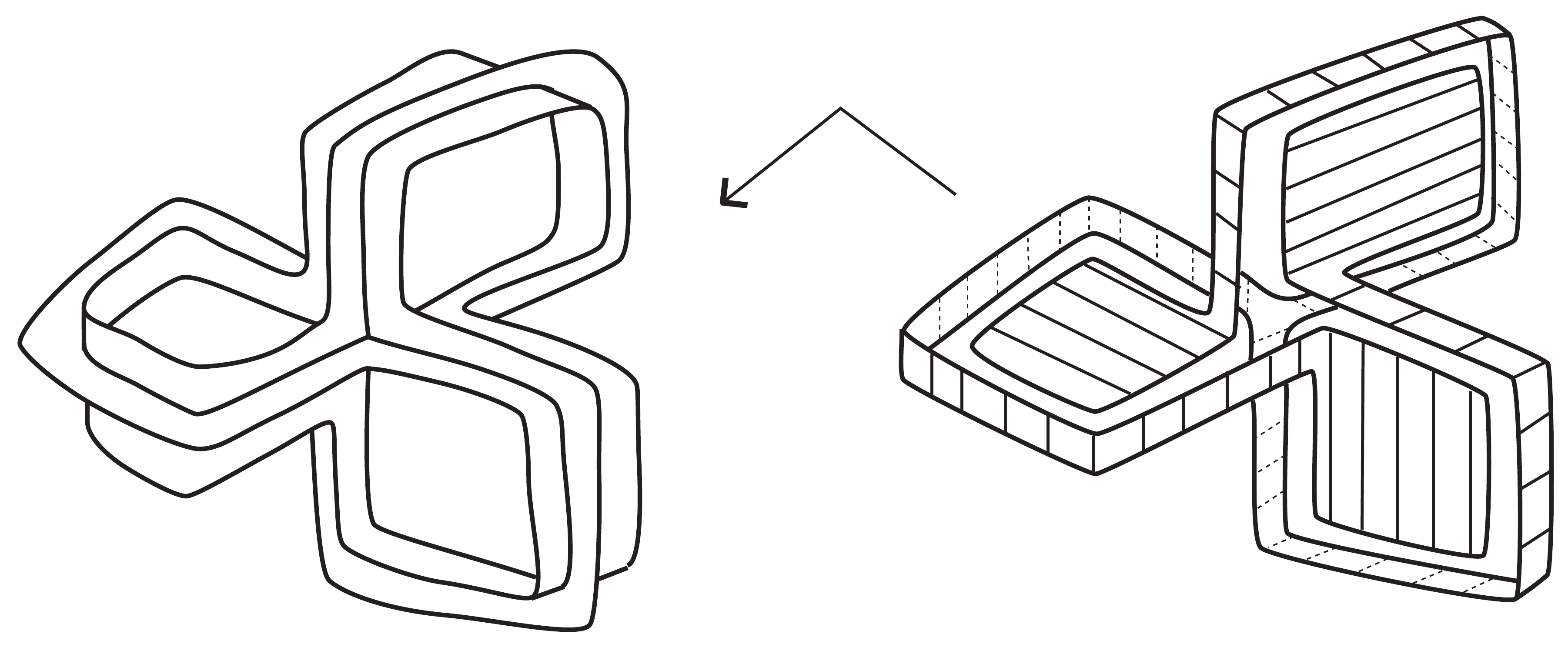}
\caption{Boy's surface $P$}
\label{Boy}
\end{figure}

We have obviously a continuous map
$h_0 \co S_0 \times [0, 1] \to \R^3$
with the following properties:
\begin{enumerate}
\item $h_0|_{S_0 \times \{0\}} = \id_{S_0} \co S_0
\times \{0\} \to S_0$.
\item $h_0|_{S_0 \times (0, 1)}$ is
a diffeomorphism onto the region bounded
by $S_0$ and $P$.
\item $h_0(S_0 \times \{1\}) = P$.
\item $h_0|_{S_0 \times \{1\}}$ is a homeomorphism
outside of a $1$--dimensional subcomplex $C_0$ of $S_0$
as depicted in \fullref{complex1}. The image
$h_0((S_0 \setminus C_0) \times \{1\})$
coincides with the complement to the multiple point
set of Boy's surface $P$ in $P$.
\end{enumerate}

\begin{figure}[ht!]
\centering
\labellist\small
\pinlabel {$C_0$} at 350 661
\endlabellist
\includegraphics[scale=0.25]{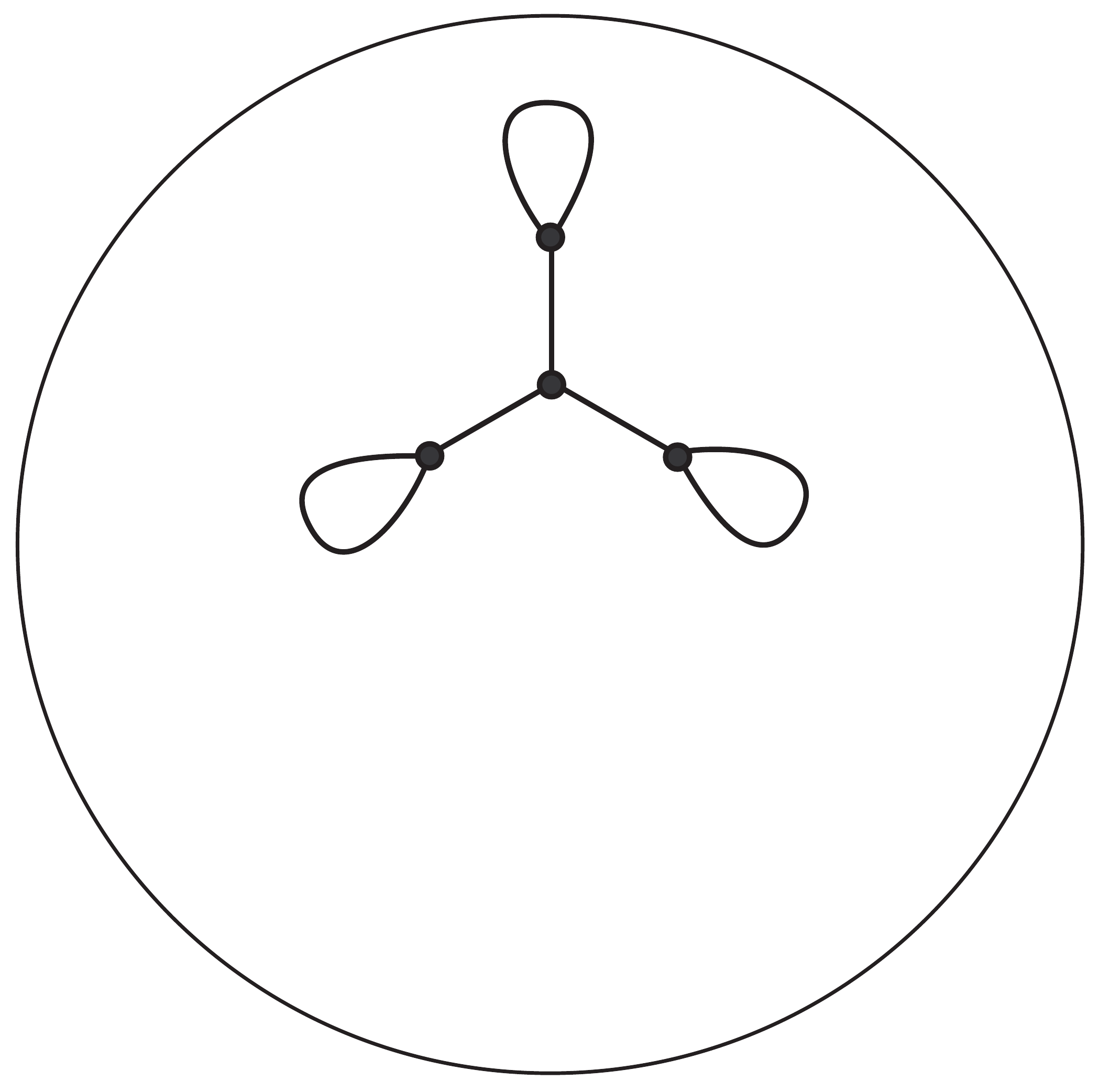}
\caption{$1$--dimensional subcomplex $C_0$ on the $2$--sphere $S_0$}
\label{complex1}
\end{figure}

Set $B = h_0(S_0 \times [0, 1/2])$.
Then $N(\tilde{S}_0) = f^{-1}(B)$
can be identified with a normal disk bundle 
$\nu_0$ to $\tilde{S}_0$ in $M$. 
In order to calculate the self-intersection
number $\tilde{S}_0 \cdot \tilde{S}_0$ in $M$,
let us consider the disk bundle $\tilde{\nu}_0$
over $\tilde{S}_0$ which is obtained
from $\nu_0$ by identifying the antipodal
points on each disk fiber. In other words,
the $S^1$--bundle $\partial \tilde{\nu}_0$
associated with the disk bundle $\tilde{\nu}_0$
corresponds to the $\R P^1$--bundle associated
with the real $2$--plane bundle afforded by $\nu_0$. 
Note that the self-intersection number
of the zero section of $\nu_0$ is equal to
one half of the self-intersection number of that
of $\tilde{\nu}_0$.

Let us construct a section of the $S^1$--bundle
$\partial \tilde{\nu}_0$
associated with $\tilde{\nu}_0$
over a certain subset $\tilde{X}_0$ of $\tilde{S}_0$.
More precisely, for each point $\tilde{x}$ of 
$\tilde{X}_0$, we
will choose a pair of antipodal points on the circle fiber
of $\partial \nu_0$ over $\tilde{x}$ continuously 
with respect to $\tilde{x}$
so that the projection restricted to
the set of all these points is a double
covering map onto $\tilde{X}_0$, where
$\partial \nu_0$ is the $S^1$--bundle associated
with the disk bundle $\nu_0$.

Let $N(C_0)$ be a regular neighborhood of $C_0$
in $S_0$. Set $\tilde{C}_0 = f^{-1}(C_0)$
and $N(\tilde{C}_0) = f^{-1}(N(C_0))$.
Note that $\tilde{C}_0$ is a $1$--dimensional
subcomplex of $\tilde{S}_0$ and that
$N(\tilde{C}_0)$ is a regular neighborhood of
$\tilde{C}_0$ in $\tilde{S}_0$.
We will first construct a
section of the $S^1$--bundle $\partial \tilde{\nu}_0$
over $\tilde{S}_0 \setminus \Int N(\tilde{C}_0)
(\subset \tilde{X}_0)$ as follows.

Take a point $\tilde{x} \in \tilde{S}_0 
\setminus \Int N(\tilde{C}_0)$ and set $x =
f(\tilde{x}) \in S_0$.
Then $h_0(x, 1) \in P$ is the image of
a unique indefinite fold point. Therefore,
$f^{-1}(h_0(x, 1/2))$ can be considered
as a nearby fiber of the fiber
over $h_0(x, 1)$, which is of
type $\mathrm{I}^1$, and hence we can
take a pair of two antipodal points 
on $f^{-1}(h_0(x, 1/2))$ 
canonically as in \fullref{two1}.
We can thus construct a
continuous section of $\partial \tilde{\nu}_0$
over each component of $\tilde{S}_0
\setminus \Int{N(\tilde{C}_0)}$.

\begin{figure}[ht!]
\centering
\labellist\small
\pinlabel {$f^{-1}(h_0(x, 1))$} [l] at 554 520
\pinlabel {$f^{-1}(h_0(x, 1/2))$} [r] at 80 542
\endlabellist
\includegraphics[scale=0.3]{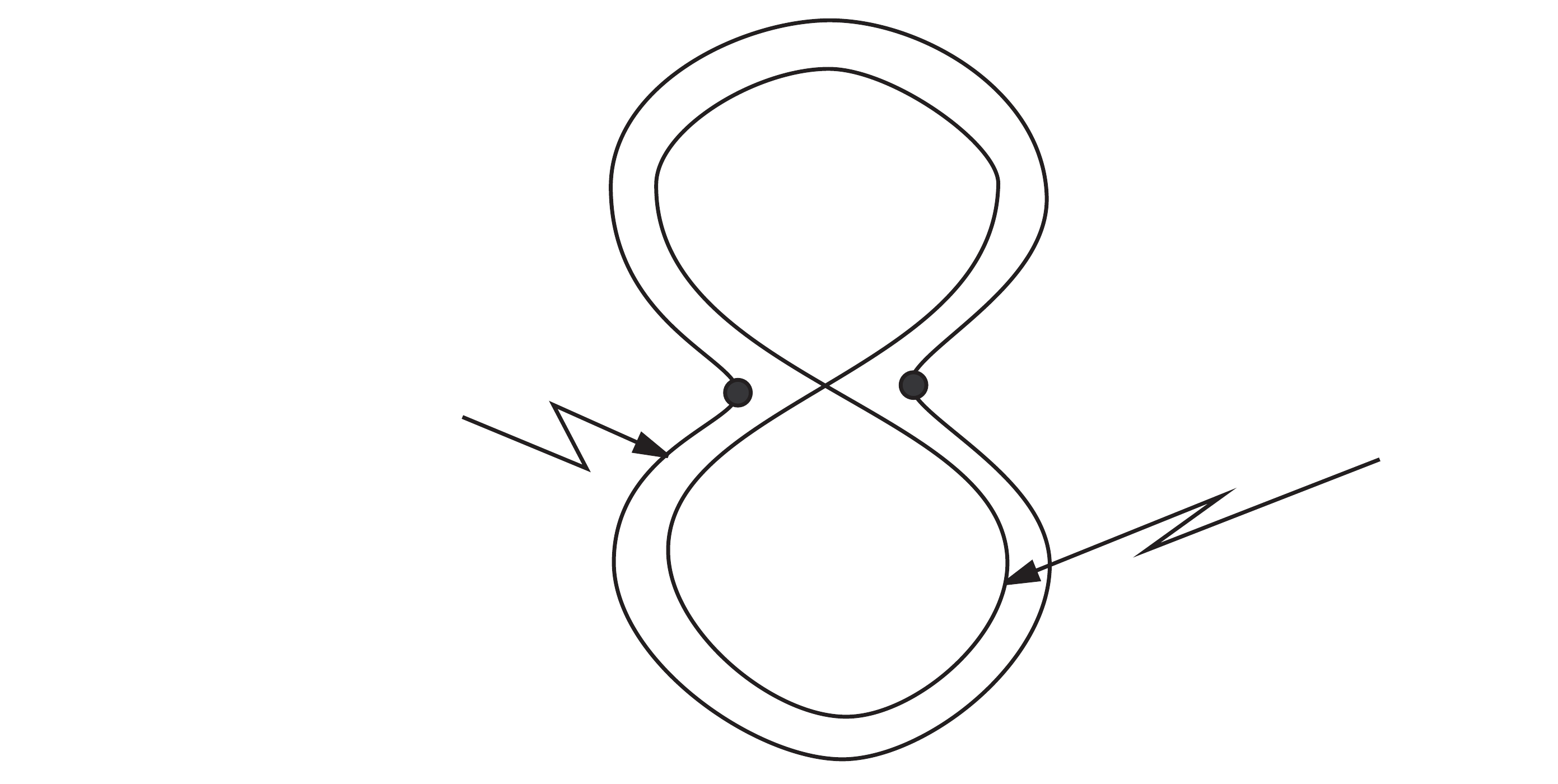}
\caption{Two points on a regular fiber near a singular fiber
of type $\mathrm{I}^1$}
\label{two1}
\end{figure}

Let us consider the twelve bands embedded
in $N(\tilde{C}_0)$ as in \fullref{nbhd1}, where
each band is homeomorphic to $[-1,1]
\times [-1,1]$. Each band is also
considered to be a $1$--handle attached
to $\tilde{S}_0 \setminus \Int{N(\tilde{C}_0)}$
along $\{-1,1\} \times [-1,1]$.
We orient the core of each band
as in the figure, where a core
is an arc properly embedded in a band corresponding to
$[-1,1] \times \{0\}$. The subset $\tilde{X}_0$
is the union of $\tilde{S}_0 \setminus \Int{N(\tilde{C})}$
and the twelve bands. 
Let us extend the
section of $\partial \tilde{\nu}_0$ 
over $\tilde{S}_0 \setminus \Int{N(\tilde{C}_0)}$
constructed above
through
the twelve bands as follows.

\begin{figure}[ht!]
\centering
\labellist\small
\pinlabel {$N(\wwtilde{C}_0)$} [l] at 408 676
\endlabellist
\includegraphics[scale=0.3]{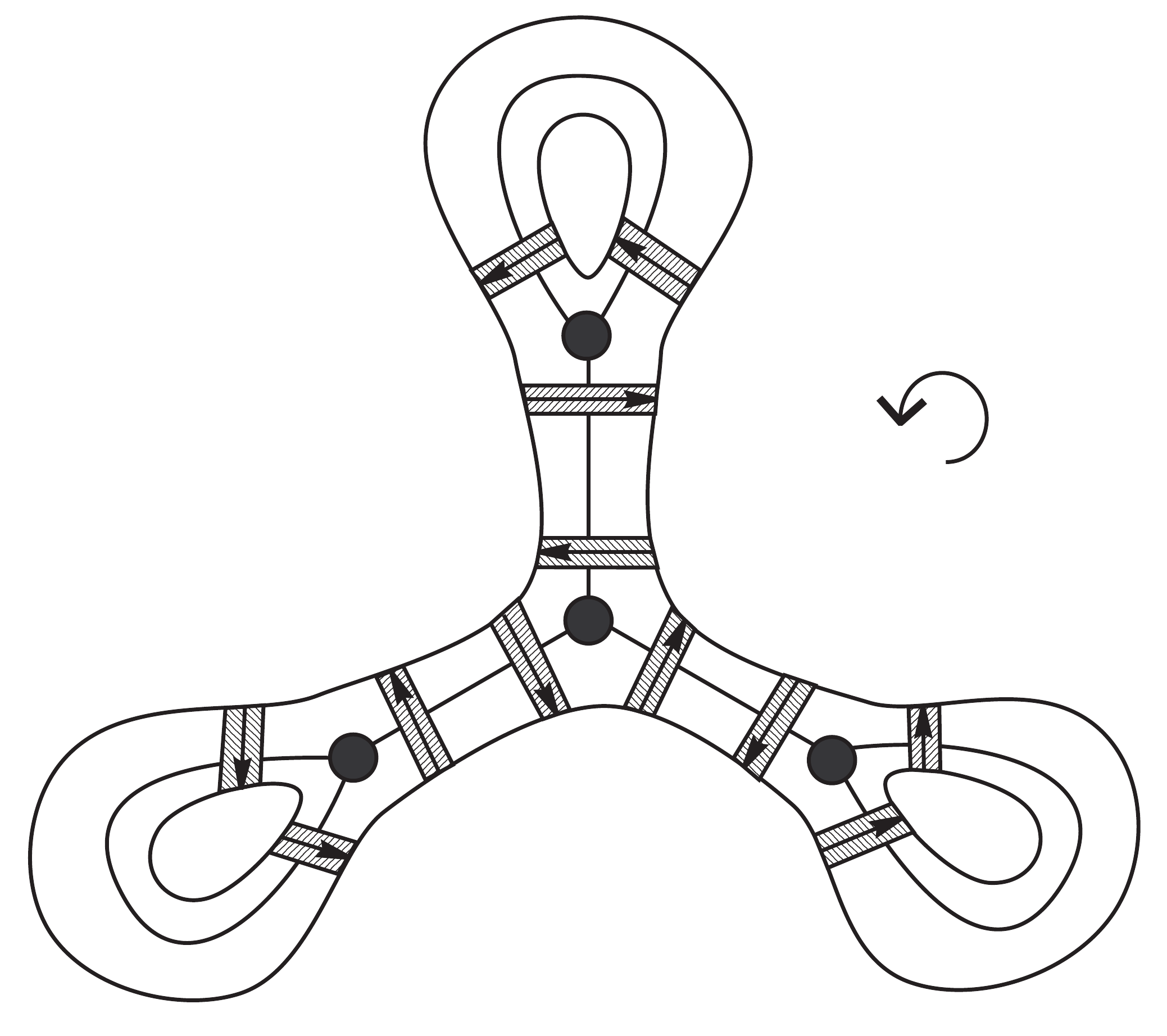}
\caption{Twelve bands in $N(\tilde{C}_0)$}
\label{nbhd1}
\end{figure}

Take a band and let $\tilde{\alpha}$ be
its core. Set $\alpha = f(\tilde{\alpha})$.
Since $\alpha' = h_0(\alpha \times \{1/2\})$
is close to the transverse double
point of $f(S(f))$ as in
\fullref{fig3} (2), the regular fibers
over the two points $\partial \alpha'$ are close to 
a $\mathrm{II}^3$--fiber. 
Furthermore, the pairs of antipodal points
on the circle fibers over the two points
$\partial \tilde{\alpha}$
associated with the above-constructed 
section of $\partial
\tilde{\nu}_0$ over 
$\tilde{S}_0 \setminus \Int{N(\tilde{C}_0)}$
are situated as in \fullref{double1}.
Let us extend the section of $\partial \tilde{\nu}_0$
over $\tilde{S}_0 \setminus \Int{N(\tilde{C}_0)}$
through $\tilde{\alpha}$ so that when one goes
along $\tilde{\alpha}$ in the direction indicated
as in \fullref{nbhd1}, the pair
of antipodal points on the circle
fibers of $\partial \nu_0$ gets the rotation
through the angle $\pi/2$ in the
positive direction of regular fibers.
Then we can naturally extend the section
to the band. We apply this construction to
all the twelve bands to get a section
of $\partial \tilde{\nu}_0$ over $\tilde{X}_0$.

\begin{figure}[ht!]
\centering
\labellist\small
\pinlabel {\rotatebox{270}{\large $=$}} at 630 457
\pinlabel {close} [t] at 335 600
\pinlabel {a $\mathrm{II}^3$--fiber} [t] at 83 520
\endlabellist
\includegraphics[scale=0.27]{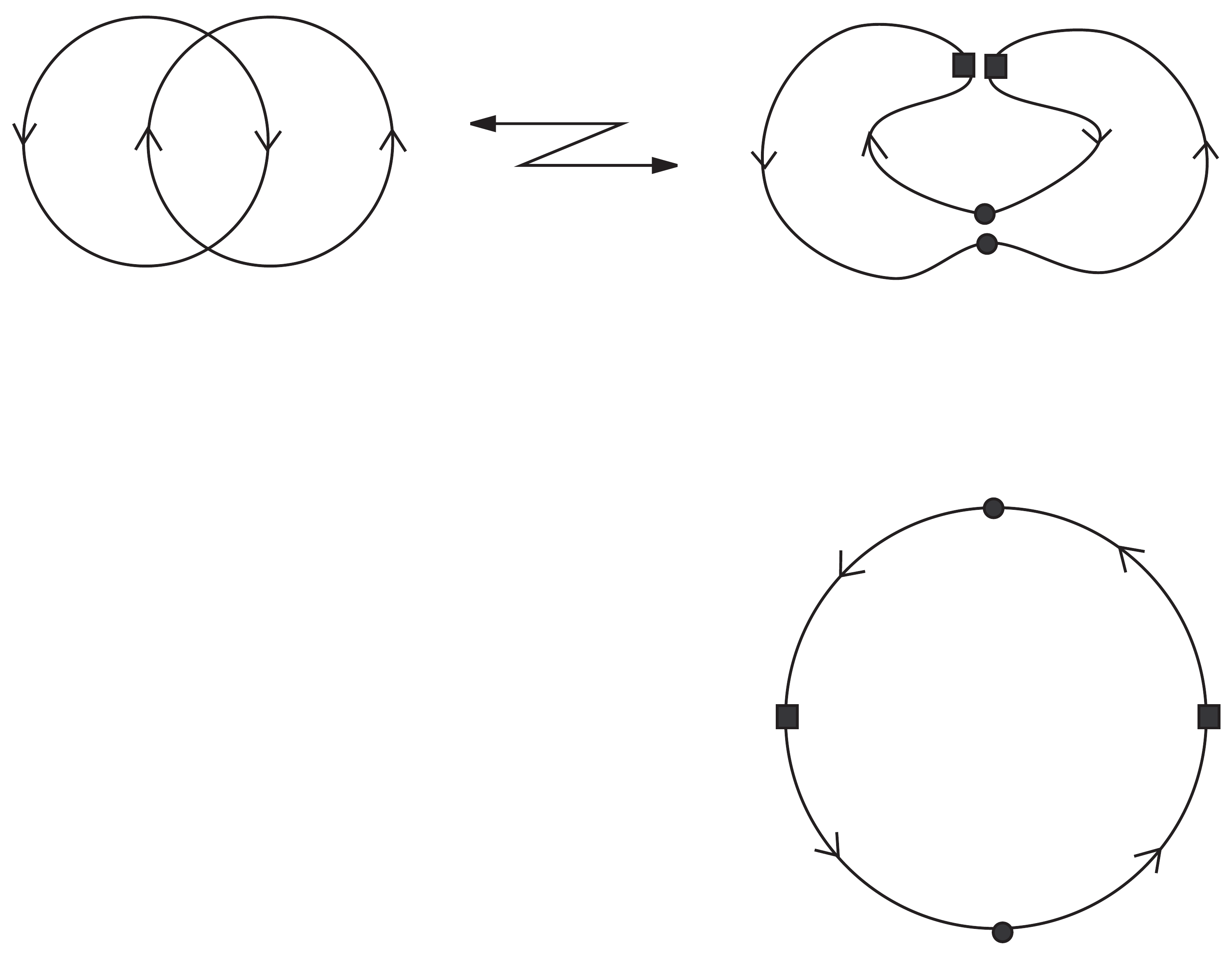}
\caption{Two pairs of antipodal points}
\label{double1}
\end{figure}

The complement of $\Int \tilde{X}_0$ in $\tilde{S}_0$
consists of ten $2$--disks: six rectangular
disks corresponding to the edges of 
$\tilde{C}_0$ and four hexagonal disks
corresponding to the vertices of $\tilde{C}_0$. 
Note that the bundle $\tilde{\nu}_0$ over
each $2$--disk $\Delta$ is trivial, and using
a trivialization, we can define the degree
of the above-constructed section over
$\partial \Delta$.
The self-intersection
number of the zero section of $\tilde{\nu}_0$
is then equal to the sum of these ``degrees"
over the boundaries of the ten $2$--disks. 
Note that here $\tilde{S}_0$ should be
oriented so that its orientation together 
with the orientation of the $2$--disk fiber
gives the orientation of the total space.
Since we have oriented the source $4$--manifold
by the ``fiber first" convention,
we may assume that 
\begin{itemize}
\item[(i)] the induced orientation on a
$2$--disk fiber of the normal disk bundle $\nu_0$
to $\tilde{S}_0$
in $M$ is given by the orientation of the 
boundary regular fiber plus the
``inward normal", and
\item[(ii)] the induced orientation
$\mathcal{O}_{S_0}$ on $S_0$ satisfies
that the ``outward normal" plus $\mathcal{O}_{S_0}$
is consistent with the right-handed orientation
of $\R^3$. 
\end{itemize}
Hence $\mathcal{O}_{S_0}$
is the ``left-handed" orientation when
viewed from outside (see \fullref{nbhd1}). 
Here we adopt the
convention that the figure
of $N(\tilde{C}_0)$ (\fullref{nbhd1})
is consistent with that of $N(C_0)$ viewed
from outside.

It is easy to see that for each of the
six rectangular $2$--disks of
$\tilde{S}_0 \setminus \Int \tilde{X}_0$, the degree
of the section of $\partial \tilde{\nu}_0$
over its boundary is equal to $-1$, since
when we go around its boundary in its
positive direction, the pair
of antipodal points on the circle
fibers of $\partial \nu_0$ rotates through the
angle $-\pi$.

Let us now consider the contribution of 
each of the other four regions
of $\tilde{S}_0 \setminus \Int \tilde{X}_0$
that are hexagonal.
The image $H$ by $h_0(\ast, 1/2)$
of its $f$--image is close to the
triple point $y$ of $f(S(f))$ and
it lies in a $1$--octant (see \fullref{triple0}).
Recall that $H$ is given the
``left-handed" orientation when
viewed from outside.
Therefore, when we go along the boundary
of the hexagonal disk $H$ in the
positive direction from a
point near the
sheet $f(D_1)$, then the second sheet
that we pass by is the sheet $f(D_2)$.
In this process, the pair of antipodal points
on a circle fiber of $\partial \nu$
corresponding to the sheet $f(D_1)$
makes a rotation through the angle
$\pi/3$ as in \fullref{three1},
since the sign of the $\mathrm{III}^8$ type
fiber is positive.
Therefore,
the degree of the section
of $\partial \tilde{\nu}_0$
over the boundary of the hexagonal disk
is equal to $+1$, since
when we go around its boundary in its
positive direction, the pair
of antipodal points on the circle fibers
of $\partial \nu_0$ rotates through the
angle $\pi$ (see \fullref{three1}). We also note that
this argument can be equally applied to all the
four hexagonal regions (that is, the central hexagonal
region is not an exception), since the argument is
local in nature in the target.

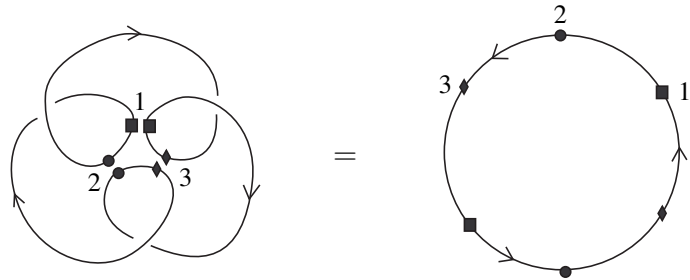
\begin{figure}[ht!]
\centering
\labellist\small
\pinlabel {\large $=$} at 270 630
\pinlabel {$1$} [b] at 12 680
\pinlabel {$2$} [tr] at -30 618
\pinlabel {$3$} [tl] at 54 624
\pinlabel {$1$} [l] at 678 714
\pinlabel {$2$} [b] at 540 790
\pinlabel {$3$} [r] at 414 723
\endlabellist
\includegraphics[scale=0.3]{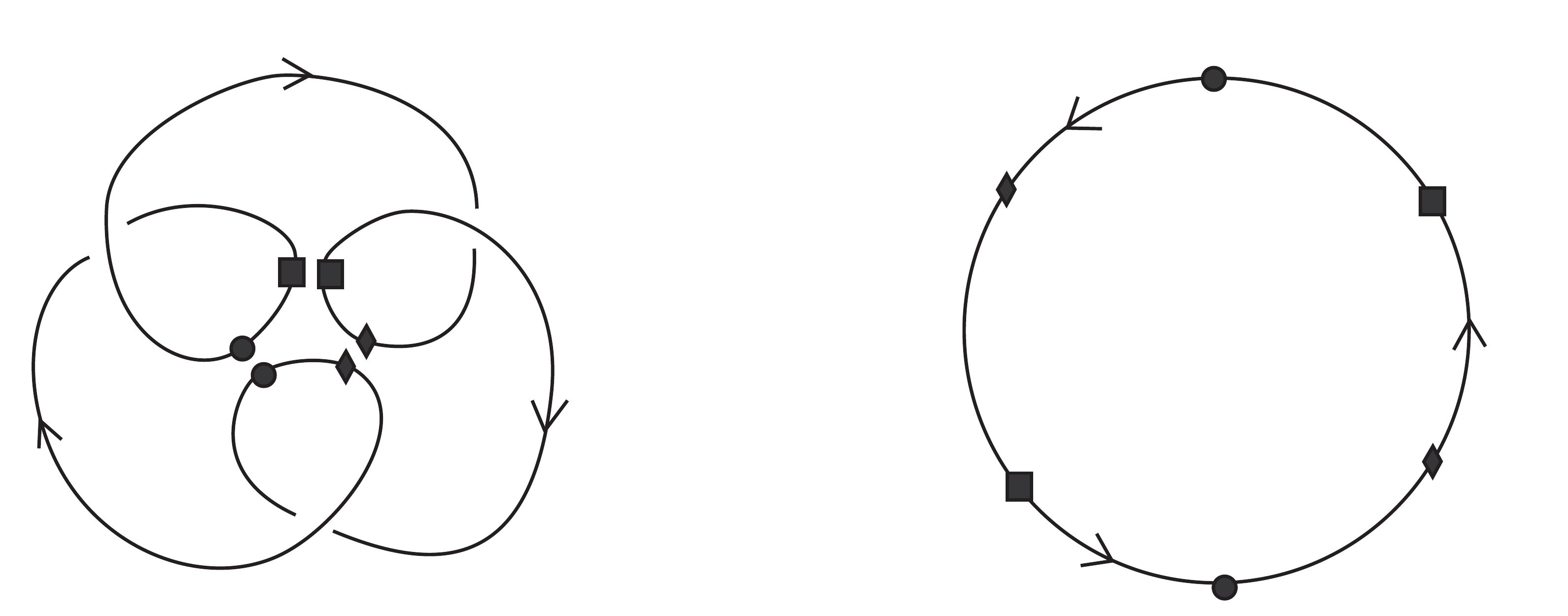}
\caption{Three pairs of antipodal points}
\label{three1}
\end{figure}

Thus the sum of the degrees is equal to
$$(-1) \cdot 6 + (+1) \cdot 4 = -2.$$
Therefore, the self-intersection number
of the zero section of $\tilde{\nu}_0$
is equal to $-2$ and that of $\nu_0$
is equal to $-2/2 = -1$.

The self-intersection
numbers $\tilde{S}_1 \cdot \tilde{S}_1$
and $\tilde{S}_2 \cdot \tilde{S}_2$
can be computed by a similar method
as follows. Let $\nu_i$ be
the normal disk bundle to $\tilde{S}_i$
in $M$, $i = 1, 2$. We can construct
the continuous map $h_i$, $i = 1, 2$,
for $f(\tilde{S}_i) = S_i$ satisfying the
properties similar to those for $h_0$.
Then we take the $1$--dimensional
complexes $C_i$ and $\tilde{C}_i$,
their regular neighborhoods $N(C_i)$
and $N(\tilde{C}_i)$ respectively, 
and the twelve bands in $\tilde{S}_i$,
$i = 1, 2$, 
as above, and define $\tilde{X}_i$ to
be the union of $\tilde{S}_i \setminus
\Int N(\tilde{C}_i)$ and the twelve
bands. Then we can construct
a section of $\partial \nu_i$, $i = 1, 2$,
over $\tilde{X}_i$ by using an argument similar
to that for $\partial \tilde{\nu}_0$.
(Note that in the present case, we
use $\nu_i$ itself, rather than its
associated $\R P^1$--bundle.)
More precisely,
let us consider the regular fiber
over the point $h_i(x, 1/2)$ for
a point $\tilde{x} \in \tilde{S}_i \setminus
\Int N(\tilde{C}_i)$ with $x = f(\tilde{x})$. 
It consists of
two circles $\ell_1$ and $\ell_2$
and we can take a pair
of points which lie on distinct
components as in
\fullref{disconnected1}.

\begin{figure}[ht!]
\centering
\labellist\small
\pinlabel {$\ell_1$} [r] at 8 687
\pinlabel {$\ell_2$} [r] at 38 526
\pinlabel {$f^{-1}(h_i(x, 1))$} [l] at 430 468
\pinlabel {$f^{-1}(h_i(x, 1/2))$} [l] at 430 590
\endlabellist
\includegraphics[scale=0.3]{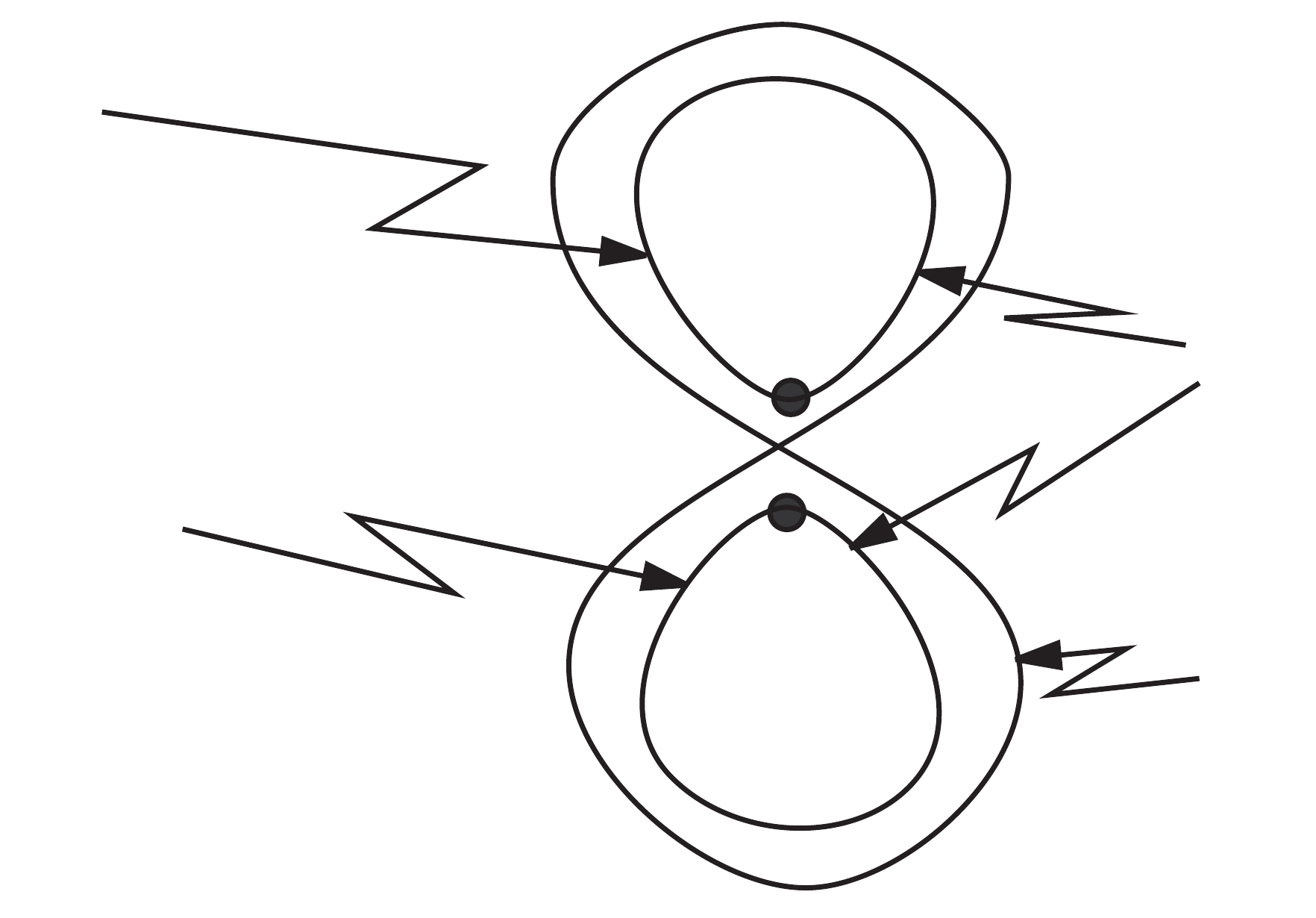}
\caption{Two points on a regular fiber which are
on distinct connected components}
\label{disconnected1}
\end{figure}
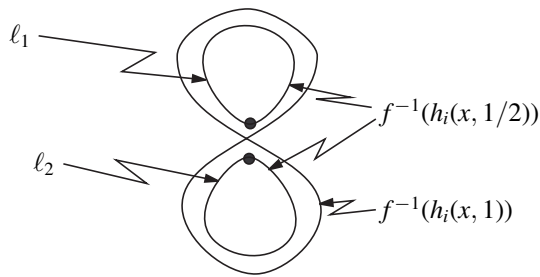

We can extend this section through
the twelve bands as in the case of
$\tilde{S}_0$. Let us compute
the sum of the degrees of the section
over the components of $\partial \tilde{X}_i$.
First note that the orientation on $S_i$
is the ``left-handed" orientation when
viewed \emph{from inside}. Then
the contribution of each rectangular
region is equal to $+1$, since
when we go along the core of a band,
the chosen point on a connected
component of a regular fiber
gets the rotation through the
angle $\pi$ (see \fullref{moon1}). 

\begin{figure}[ht!]
\centering
\labellist\small
\pinlabel {close} [t] at 304 595
\pinlabel {a $\mathrm{II}^3$--fiber} [t] at 10 500
\endlabellist
\includegraphics[scale=0.3]{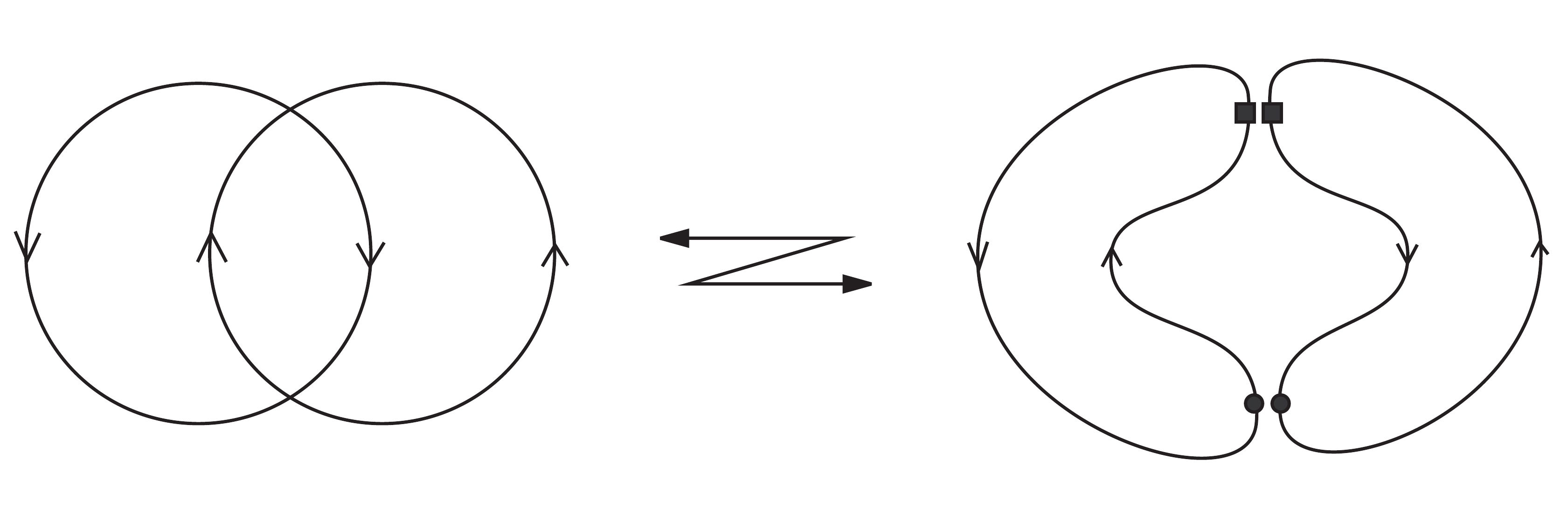}
\caption{Two pairs of points}
\label{moon1}
\end{figure}

As to the other four hexagonal
regions, they correspond to
the triple point of $f(S(f))$.
When we go around the
boundary, the chosen point gets
the rotation through the angle $-2\pi$
(see \fullref{triangle1}).
Hence its contribution to the
self-intersection number is equal to $-1$.

\begin{figure}[ht!]
\centering
\labellist\small
\pinlabel {$1$} [b] at 309 588
\pinlabel {$2$} [r] at 238 488
\pinlabel {$3$} [l] at 391 485
\endlabellist
\includegraphics[scale=0.3]{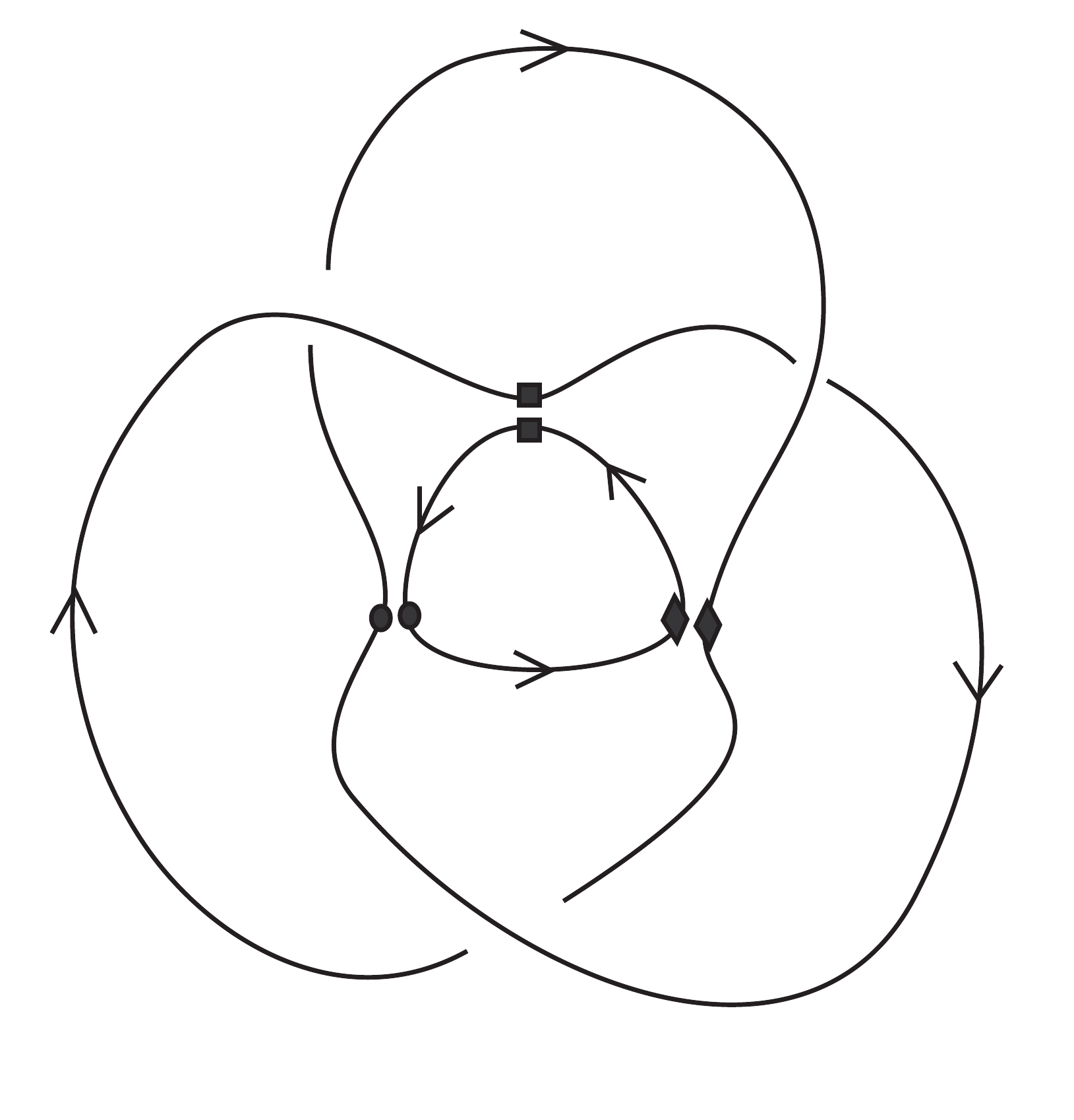}
\caption{Three pairs of points}
\label{triangle1}
\end{figure}

Therefore, the self-intersection number
$\tilde{S}_i \cdot \tilde{S}_i$ is equal to
$$(+1) \cdot 6 + (-1) \cdot 4 = 2$$
for $i = 1, 2$.

Thus the self-intersection number of
$S_0(f)$ in $M$ is equal to
$$S_0(f) \cdot S_0(f) =
\tilde{S}_0 \cdot \tilde{S}_0 + 
\tilde{S}_1 \cdot \tilde{S}_1 + 
\tilde{S}_2 \cdot \tilde{S}_2 = -1 + 2 + 2 = 3.$$
Therefore, the signature of the source
$4$--manifold is equal to $+1$ according to
the formula \eqref{eq:self-intersection}.
In other words, the source $4$--manifold
$M$ is oriented diffeomorphic to
$2 \C P^2 \sharp \overline{\C P^2}$.
This completes the proof of
\fullref{lem:example}.
\end{proof}

Let us now proceed to the proof of our
main theorem of this paper.

\begin{proof}[Proof of \fullref{thm:main}]
Let us fix a $3$--manifold $N$.
For a closed oriented $4$--manifold $M$ and 
a $C^\infty$ stable map $f \co M \to N$
into $N$, let us consider the integer
$$\psi(M, f) = ||\mathrm{III}^8(f)||.$$
By virtue of \fullref{prop:cobord-inv},
$\psi$ depends only on the oriented
cobordism class of $M$. Therefore, $\psi$
induces a map
$$\overline{\psi} \co \Omega_4 \to \Z$$
of the $4$--dimensional oriented
cobordism group to the additive group of integers.
This is clearly a homomorphism. 

Recall that $\Omega_4$ is an infinite
cyclic group generated by the class of
oriented $4$--manifolds with signature $+1$.
In other words, the signature induces
an isomorphism 
$$\sigma \co \Omega_4 \to \Z.$$

Let us consider the explicit example
$f \co 2 \C P^2 \sharp \overline{\C P^2}
\to \R^3$ given in \fullref{lem:example}.
By composing it with an embedding $\R^3 \hookrightarrow N$,
we get a $C^\infty$ stable map of an
oriented $4$--manifold of signature $+1$
into $N$.
This stable map has exactly one $\mathrm{III}^8$ type fiber,
whose sign is equal to $+1$.
Therefore, the homomorphism
$$\overline{\psi} \circ \sigma^{-1} \co \Z \to \Z$$
sends $+1$ to $+1$ and hence is the identity.
Thus we have $\sigma = \overline{\psi}$.
This completes the proof.
\end{proof}

We have a direct consequence of our main
theorem as follows.

\begin{cor}
Let $M$ be a closed oriented $4$--manifold
and $N$ a $3$--manifold. Then every
$C^\infty$ stable map of $M$ into $N$
has at least $|\sigma(M)|$ singular fibers
of type $\mathrm{III}^8$.
\end{cor}

For example, if $M$ is the underlying
real $4$--dimensional manifold of the
complex K3 surface, then every
$C^\infty$ stable map of $M$ into
a $3$--manifold has at least $16$
fibers of type $\mathrm{III}^8$,
although no explicit example
of such a map is known.
Construction of an explicit example
can be an interesting problem.

Our study of the explicit example
gives a new proof to the following
formula, which has been proved
by Ohmoto--Saeki--Sakuma \cite{OSS} and Sadykov \cite{Sadykov}.

\begin{cor}\label{cor:int}
Let $f \co M \to N$ be a 
$C^\infty$ stable map of a closed
oriented $4$--manifold $M$ into an orientable
$3$--manifold $N$. Then we have
$$S(f) \cdot S(f) = 3 \sigma(M),$$
where $S(f) \cdot S(f)$ is the self-intersection
number of the surface of singular point set
of $f$ in $M$.
\end{cor}

Note that the formula \eqref{eq:self-intersection}
follows from \fullref{cor:int}.
We give a proof to the above corollary
without using the formula \eqref{eq:self-intersection}.

\begin{proof}[Proof of \fullref{cor:int}]
It is not difficult to show that the
self-intersection number $S(f) \cdot S(f)$
is an oriented cobordism invariant of the
source $4$--manifold $M$ (for this, use
the argument as in \cite[Proof of Lemma~3.2]{Sadykov}
together with a result of Conner--Floyd \cite{CF} as in the
proof of \fullref{prop:cobord-inv}).
Therefore, there exists a constant $c$
such that
$$S(f) \cdot S(f) = 3 c \sigma(M)$$
holds for any $C^\infty$ stable map $f$
of a closed oriented $4$--manifold $M$
into an orientable $3$--manifold $N$.

For the explicit example $f$ studied above,
we have\footnote{We know that
$S(f) \cdot S(f) = 3 \sigma(M) = 3$:
however, in order to show this, we used the formula
\eqref{eq:self-intersection}. Here we are
proving the corollary without
using \eqref{eq:self-intersection}.} 
$S(f) \cdot S(f) = \pm 3$ and
$3 \sigma(M) = \pm 3$. Therefore, the
constant $c$ must be equal to $\pm 1$.

On the other hand, by a result
of Sakuma \cite{Sakuma93}
we have
$$S(f) \cdot S(f) \equiv 3 \sigma(M) \pmod{4}.$$
Therefore, the constant $c$ must be equal
to $+1$. 
This completes the proof.
\end{proof}

Recall that Sadykov \cite{Sadykov}
proved \fullref{cor:int}
by a characteristic class argument
together with Sakuma's result \cite{Sakuma93}.

Now some remarks concerning our
result are in order.

\begin{rem}
For a closed oriented $4$--manifold $M$, let
$n_+$ and $n_-$ be arbitrary integers such that
$\sigma(M) = n_+ - n_-$. Then does there
exist a $C^\infty$ stable map $f \co M \to N$
that has exactly $n_+$ positive singular fibers
of $\mathrm{III}^8$ type and $n_-$ negative ones?
The authors do not know the answer to the question.
\end{rem}

\begin{rem}
In \cite{Saeki93} the first named author
proved that if $f \co M \to N$ is a $C^\infty$
stable map of a closed oriented $4$--manifold
into a $3$--manifold with only definite
fold singularities, then $\sigma(M) = 0$.
This follows from our main theorem
as well.
\end{rem}

\begin{rem}\label{rem:direct}
The technique of this section 
to determine 
the self-intersection number of the
surface of definite fold points may be
generalized in a more general setting.
This might give a new direct proof
of our main theorem. 
\end{rem}

\begin{rem}\label{rem:Lef}
It is known that there exist oriented
surface bundles over oriented surfaces
with non-zero signatures (for example,
see Meyer \cite{Meyer}). This means that
we cannot expect a similar signature
formula for $C^\infty$ stable maps
of closed oriented $4$--manifolds into
surfaces. 

Note that if the fiber genus
$g \leq 2$, then there is a
signature formula for Lefschetz fibrations
in terms of singular fibers (see the work of Matsumoto
\cite{Matsumoto1,Matsumoto2}). 
We also have a similar formula for hyperelliptic
Lefschetz fibrations of any genus (see Endo \cite{Endo}).
It may be possible to prove these formulas
by using our main theorem as follows.
Let $M$ be a Lefschetz fibration over a
surface $B$, and
consider a line bundle $N$ over $B$.
Note that $N$ is a $3$--manifold.
Then we can consider a generic map $f \co M
\to N$ which makes the diagram
$$\bfig
  \barrsquare/->`->`->`=/<500,300>[M`N`B`B;f```]
  \efig$$
commutative, where the vertical arrows are relevant
fibration maps. In other words, we consider
a generic family of fiberwise functions.
By applying our \fullref{thm:main} to
$f$, we might be able to get a signature
formula for certain Lefschetz fibrations.
\end{rem}

\section{Universal complex of chiral singular fibers}
\label{section8}

In \cite{Saeki04}, the universal complex
of singular fibers was introduced 
as a refinement of Vassiliev's universal
complex of multi-singularities
(see Vassiliev \cite{V}, Kazaryan \cite{Ka} or Ohmoto \cite{Ohmoto}), and
it was shown that its cohomology classes
give invariants of cobordisms of singular maps
in the sense of Rim\'anyi and Sz\H{u}cs \cite{RS}.
In this section, we study the universal
complex (with integer coefficients)
of singular fibers corresponding to
chiral singular fibers and give
an interpretation of our main theorem
in terms of the theory of universal complex
of singular fibers.

We can define the universal complex of
chiral singular fibers
for proper $C^\infty$ stable maps of
oriented $5$--manifolds into $4$--manifolds
by exactly the same
procedure as in \cite{Saeki04}
as follows. 

For $\kappa$ with $3 \leq \kappa \leq 4$, let
$C^\kappa$ be the free $\Z$--module
generated by the $C^0$ equivalence classes
modulo regular fibers
of chiral singular fibers of codimension
$\kappa$. Note that $\rank C^3 = 3$ and
$\rank C^4 = 14$
according to \fullref{prop:5chiral}.
Since there exist
no chiral singular fibers of codimension
$\kappa \neq 3, 4$, we put
$C^\kappa = 0$ for $\kappa \neq 3, 4$.
Note that for $\kappa = 4$, we take
the $C^0$ equivalence classes modulo
regular fibers of chiral
singular fibers with positive signs as generators,
and we consider those with negative
signs to be $-1$ times the corresponding
class with positive sign.

The coboundary homomorphism
$\delta_3 \co C^3 \to C^4$ is defined by
$$\delta_3(\mathfrak{G}) = \sum_{\kappa(\mathfrak{F}) = 4}
[\mathfrak{G}: \mathfrak{F}]\mathfrak{F}$$
for every generator $\mathfrak{G}$ of $C^3$,
where $[\mathfrak{G}: \mathfrak{F}] \in \Z$
is the incidence coefficient which can
be defined by exactly the same method
as for $[\mathrm{III}^8: \mathfrak{F}]$
(see \fullref{section6}).
Note that all the other coboundary
homomorphisms $\delta_\kappa$, $\kappa \neq
3$, are necessarily trivial.

We call the resulting cochain complex 
$(C^\kappa, \delta_\kappa)_\kappa$
the \emph{universal complex of
chiral singular fibers} for proper
$C^\infty$ stable maps of oriented $5$--manifolds
into $4$--manifolds. Note that its
unique cohomology group that makes sense
is its third cohomology group, and
is nothing but the kernel of the coboundary
homomorphism $\delta_3$.

Then we get the following.

\begin{prop}\label{prop:univ}
The $3$--dimensional cohomology group
of the universal complex of chiral singular fibers
for proper $C^\infty$ stable maps of oriented
$5$--manifolds into $4$--manifolds
is an infinite cyclic group
generated by the $C^0$ equivalence
class modulo regular fibers of $\mathrm{III}^8$ type fibers.
\end{prop}

\begin{proof}
Recall that we have exactly three $C^0$
equivalence classes modulo regular fibers
of chiral singular fibers of codimension $3$,
namely, $\mathrm{III}^5$, $\mathrm{III}^7$ and
$\mathrm{III}^8$, by 
\fullref{prop:5chiral}.
By \fullref{lem:incidence}, the incidence
coefficients involving $\mathrm{III}^8$
are all zero and hence $\mathrm{III}^8$
is a cocycle. On the other hand, for the
other two equivalence classes of chiral
singular fibers, we have, for example,
\begin{eqnarray*}
& & [\mathrm{III}^5: \mathrm{IV}^{11}] \neq 0,
\quad 
[\mathrm{III}^5: \mathrm{IV}^{10}] = 0, \\
& & [\mathrm{III}^7: \mathrm{IV}^{11}] = 0,
\quad
[\mathrm{III}^7: \mathrm{IV}^{10}] \neq 0.
\end{eqnarray*}
Therefore, 
a linear combination of 
$\mathrm{III}^5$, $\mathrm{III}^7$ and
$\mathrm{III}^8$ is a cocycle if
and only if the coefficients of
$\mathrm{III}^5$ and $\mathrm{III}^7$
both vanish. Therefore, the
kernel of the coboundary homomorphism $\delta_3$
is infinite cyclic and is generated
by $\mathrm{III}^8$.
This completes the proof.
\end{proof}

According to \fullref{prop:univ}, we
can interpret our main theorem (\fullref{thm:main})
as follows. The $3$--dimensional cohomology
class represented by the cocycle
$\mathrm{III}^8$ of the universal
complex of chiral singular fibers for
proper $C^\infty$ stable maps of oriented
$5$--manifolds into $4$--manifolds
gives a complete invariant of the oriented
cobordism class of the source $4$--manifold.
In particular, for $N = \R^3$,
it gives a complete invariant
of the oriented bordism class
of a $C^\infty$ stable map of a closed
oriented $4$--manifold into $\R^3$.

For related discussions, see \cite{Saeki04}.

We also see that the fiber which
satisfies the property as in 
\fullref{thm:main} should necessarily
be the fiber of type $\mathrm{III}^8$.
This explains the reason why the $\mathrm{III}^8$
type fiber
appeared in the modulo two Euler characteristic
formula in \cite{Saeki04} (see \fullref{cor:Euler}
of the present paper).

\begin{rem}
If we can realize the proof of
our main theorem as mentioned in \fullref{rem:direct}
for arbitrary stable maps of closed oriented null-cobordant
$4$--manifolds into $3$--manifolds,
then that would imply that
$\mathrm{III}^8$ is a cocycle
of the universal complex
(see \cite[Section~12.2]{Saeki04}).
That is, it might be possible
to prove that $\mathrm{III}^8$
is a cocycle without even classifying
the singular fibers.
\end{rem}

\bibliographystyle{gtart}
\bibliography{link}

\begin{thebibliography}{}
\providecommand\bibmarginpar{\leavevmode\marginpar}
\def\urlstyle#1{{\tt #1}}

\bibitem{Ando}
\textbf{Y Ando}, \href{http://dx.doi.org/10.2307/2154132} {\emph{On local
  structures of the singularities $A_k$, $D_k$ and $E_k$ of smooth maps}},
  Trans. Amer. Math. Soc. 331 (1992) 639--651 \xox{MR}{1055564}

\bibitem{BJ}
\textbf{T Br{\"o}cker}, \textbf{K J{\"a}nich}, \emph{Introduction to
  differential topology}, Cambridge University Press, Cambridge (1982)
  \xox{MR}{674117}

\bibitem{CF}
\textbf{P\,E Conner}, \textbf{E\,E Floyd}, \emph{Differentiable periodic maps},
  Ergebnisse series 33, Springer, Berlin (1964) \xox{MR}{0176478}

\bibitem{Damon}
\textbf{J Damon}, \href{http://dx.doi.org/10.1017/S0305004100058369}
  {\emph{Topological properties of real simple germs, curves, and the nice
  dimensions $n>p$}}, Math. Proc. Cambridge Philos. Soc. 89 (1981) 457--472
  \xox{MR}{602300}

\bibitem{E}
\textbf{C Ehresmann}, \emph{Sur les espaces fibr\'es diff\'erentiables}, C. R.
  Acad. Sci. Paris 224 (1947) 1611--1612 \xox{MR}{0020774}

\bibitem{Endo}
\textbf{H Endo}, \href{http://dx.doi.org/10.1007/s002080050012} {\emph{Meyer's
  signature cocycle and hyperelliptic fibrations}}, Math. Ann. 316 (2000)
  237--257 \xox{MR}{1741270}

\bibitem{GWDL}
\textbf{C\,G Gibson}, \textbf{K Wirthm{\"u}ller}, \textbf{A\,A du~Plessis},
  \textbf{E\,J\,N Looijenga}, \emph{Topological stability of smooth mappings},
  Lecture Notes in Mathematics 552, Springer, Berlin (1976) \xox{MR}{0436203}

\bibitem{GG}
\textbf{M Golubitsky}, \textbf{V Guillemin}, \emph{Stable mappings and their
  singularities}, Graduate Texts in Mathematics 14, Springer, New York (1973)
  \xox{MR}{0341518}

\bibitem{Ka}
\textbf{M\,{\`E} Kazaryan},
  \href{http://dx.doi.org/10.1070/SM1995v186n12ABEH000094} {\emph{Hidden
  singularities and {V}assiliev's homology complex of singularity classes}},
  Mat. Sb. 186 (1995) 119--128 \xox{MR}{1376094}\ \ English translation in Sb.
  Math. 186 (1995) 1811--1820

\bibitem{KLP}
\textbf{L Kushner}, \textbf{H Levine}, \textbf{P Porto}, \emph{Mapping
  three-manifolds into the plane I}, Bol. Soc. Mat. Mexicana $(2)$ 29 (1984)
  11--33 \xox{MR}{790729}

\bibitem{L}
\textbf{K Lamotke}, \href{http://dx.doi.org/10.1016/0040-9383(81)90013-6}
  {\emph{The topology of complex projective varieties after S Lefschetz}},
  Topology 20 (1981) 15--51 \xox{MR}{592569}

\bibitem{Levine1}
\textbf{H Levine}, \emph{Classifying immersions into {$\mathbf{R}\sp 4$} over
  stable maps of {$3$}-manifolds into {$\mathbf{R}\sp 2$}}, Lecture Notes in
  Mathematics 1157, Springer, Berlin (1985) \xox{MR}{814689}

\bibitem{MatherV}
\textbf{J\,N Mather},
  \href{http://www.numdam.org/item?id=PMIHES_1969__37__223_0} {\emph{Stability
  of $C^{\infty}$ mappings V: Transversality}}, Advances in Math. 4 (1970)
  301--336 (1970) \xox{MR}{0275461}

\bibitem{MatherVI}
\textbf{J\,N Mather}, \emph{Stability of $C^{\infty}$ mappings VI: The nice
  dimensions}, from: ``Proceedings of Liverpool Singularities Symposium I
  (1969/70)'', Lecture Notes in Mathematics 192, Springer, Berlin (1971)
  207--253 \xox{MR}{0293670}

\bibitem{Matsumoto1}
\textbf{Y Matsumoto},
  \href{http://projecteuclid.org/getRecord?id=euclid.pja/1195515921} {\emph{On
  4--manifolds fibered by tori II}}, Proc. Japan Acad. Ser. A Math. Sci. 59
  (1983) 100--103 \xox{MR}{711307}

\bibitem{Matsumoto2}
\textbf{Y Matsumoto}, \emph{Lefschetz fibrations of genus two -- a topological
  approach}, from: ``Topology and Teichm\"uller spaces (Katinkulta, 1995)'',
  World Sci. Publishing, River Edge, NJ (1996)  123--148 \xox{MR}{1659687}

\bibitem{Meyer}
\textbf{W Meyer}, \href{http://dx.doi.org/10.1007/BF01427946} {\emph{Die
  Signatur von Fl\"achenb\"undeln}}, Math. Ann. 201 (1973) 239--264
  \xox{MR}{0331382}

\bibitem{Ohmoto}
\textbf{T Ohmoto}, \emph{Vassiliev complex for contact classes of real smooth
  map-germs}, Rep. Fac. Sci. Kagoshima Univ. Math. Phys. Chem. 27 (1994) 1--12
  \xox{MR}{1341346}

\bibitem{OSS}
\textbf{T Ohmoto}, \textbf{O Saeki}, \textbf{K Sakuma},
  \href{http://dx.doi.org/10.1090/S0002-9947-03-03345-2}
  {\emph{Self-intersection class for singularities and its application to fold
  maps}}, Trans. Amer. Math. Soc. 355 (2003) 3825--3838 \xox{MR}{1990176}

\bibitem{RS}
\textbf{R Rim{\'a}nyi}, \textbf{A Sz{\"u}cs},
  \href{http://dx.doi.org/10.1016/S0040-9383(97)00093-1}
  {\emph{Pontrjagin--{T}hom-type construction for maps with singularities}},
  Topology 37 (1998) 1177--1191 \xox{MR}{1632908}

\bibitem{Sadykov}
\textbf{R Sadykov}, \href{http://dx.doi.org/10.1016/j.topol.2004.04.006}
  {\emph{Elimination of singularities of smooth mappings of 4--manifolds into
  3--manifolds}}, Topology Appl. 144 (2004) 173--199 \xox{MR}{2097135}

\bibitem{Saeki93}
\textbf{O Saeki}, \href{http://dx.doi.org/10.1016/0166-8641(93)90116-U}
  {\emph{Topology of special generic maps of manifolds into Euclidean spaces}},
  Topology Appl. 49 (1993) 265--293 \xox{MR}{1208678}

\bibitem{Saeki02}
\textbf{O Saeki}, \href{http://dx.doi.org/10.1007/s00014-003-0758-9}
  {\emph{Fold maps on 4--manifolds}}, Comment. Math. Helv. 78 (2003) 627--647
  \xox{MR}{1998397}

\bibitem{Saeki04}
\textbf{O Saeki}, \emph{Topology of singular fibers of differentiable maps},
  Lecture Notes in Mathematics 1854, Springer, Berlin (2004) \xox{MR}{2106689}

\bibitem{Sakuma93}
\textbf{K Sakuma}, \href{http://dx.doi.org/10.1016/0166-8641(93)90024-8}
  {\emph{On special generic maps of simply connected $2n$--manifolds into
  $\mathbb{R}^3$}}, Topology Appl. 50 (1993) 249--261 \xox{MR}{1227553}

\bibitem{V}
\textbf{V\,A Vassiliev}, \emph{Lagrange and Legendre characteristic classes},
  Advanced Studies in Contemporary Mathematics 3, Gordon and Breach Science
  Publishers, New York (1988) \xox{MR}{1065996}

\bibitem{YT2}
\textbf{T Yamamoto}, \emph{Classification of singular fibres of stable maps of
  4--manifolds into 3--manifolds and its applications}, J. Math. Soc. Japan (to
  appear)

\bibitem{YT1}
\textbf{T Yamamoto}, \emph{Classification of singular fibers and its
  applications}, Master's thesis, Hokkaido University (2002)\ (in Japanese)

\end{thebibliography}

\end{document}